\documentclass{amsart}

\newtheorem{theorem}{Theorem}[section]
\newtheorem{lemma}[theorem]{Lemma}

\newtheorem{corollary}[theorem]{Corollary}
\newtheorem{proposition}[theorem]{Proposition}

\newtheorem{conjecture}[theorem]{Conjecture}
\newtheorem{condition}[theorem]{Condition}

\theoremstyle{definition}
\newtheorem{definition}[theorem]{Definition}

\theoremstyle{remark}
\newtheorem{remark}[theorem]{Remark}

\DeclareMathOperator{\Vol}{Vol}

\begin{document}
\title{Collapsing of Calabi-Yau manifolds and  special lagrangian submanifolds  }
\author[Y. Zhang]{Yuguang Zhang }
\address{Department of Mathematics, Capital Normal University,
Beijing, P.R.China  }
 \email{yuguangzhang76@yahoo.com}
 \thanks{  Partially supported by  the National Natural Science Foundation of China
09221010054.  }

\begin{abstract} In this paper, the relationship between the existence of special lagrangian
   submanifolds and the collapsing  of   Calabi-Yau manifolds is
   studied.
 First, special lagrangian
 fibrations are constructed  on some  regions  of bounded curvature and sufficiently
 collapsed in Ricci-flat Calabi-Yau manifolds. Then, in the opposite direction,   it is shown that
  the existence of special
       lagrangian submanifolds  with small volume  implies the collapsing of some regions in  the
       ambient Calabi-Yau manifolds.
\end{abstract}
\maketitle
\section{Introduction}
The notion of special lagrangian submanifold was     introduced by
Harvey and Lawson in the seminar paper \cite{HL}. Mclean studied the
deformation theory of special lagrangian submanifolds in \cite{Mc}.
In the pioneer work \cite{SYZ}, Stominger, Yau and Zaslow propose a
conjecture about constructing the mirror manifold of a given
Calabi-Yau manifold, the SYZ conjecture,  via  special lagrangian
fibrations.  Since then, lots of works were  devoted to study
special lagrangian submanifolds and fibrations (c.f. \cite{Hit},
\cite{Ru}, \cite{Ru2}, \cite{Ru3}, \cite{Go0}, \cite{Go}, \cite{Lu},
\cite{Gro1}, \cite{Sa}, \cite{TY}, \cite{J2}, \cite{J3}, and
references in \cite{J3}).    In
    \cite{KS} and \cite{GW2},  a refined version of SYZ conjecture
    was proposed by using
   the collapsing  of  Ricci-flat Calabi-Yau manifolds in the
   Gromov-Hausdorff sense.  These two versions of SYZ conjecture
   suggest a  relationship between the existence of special lagrangian
   submanifolds and the collapsing  of   Calabi-Yau manifolds.   In this paper,
   we study this relationship.

If $(M,\omega, J,g)$ is a compact Ricci-flat  K\"{a}hler
$n$-manifold, and admits a no-where vanishing holomorphic $n$-form $
\Omega$, the holomorphic volume form, $(M,\omega, J,g, \Omega)$ is
called a Ricci-flat  Calabi-Yau $n$-manifold, and $(\omega, J,g,
\Omega)$ is called a Calabi-Yau structure on $M$.  We can  normalize
$\Omega$ such that
$$\frac{\omega^{n}}{n!}=\frac{(-1)^{\frac{n^{2}}{2}}}{2^{n}}\Omega\wedge
\overline{\Omega},$$ (c.f. \cite{J3}).  Yau's theorem of Calabi
conjecture guarantees the existence of Ricci-flat  K\"{a}hler
metrics on K\"{a}hler manifolds  with  trivial  canonical bundle
(c.f. \cite{Ya1}),
 which  implies the  existence of Calabi-Yau structures on such
 manifolds.   The holonomy group of a Ricci-flat  Calabi-Yau $n$-manifold is
a subgroup of $SU(n)$.  The study of  Calabi-Yau manifolds  is
important in both mathematics and physics (c.f. \cite{Y4}).

A special lagrangian submanifold $L$ of phase $\theta\in \mathbb{R}$
in a Ricci-flat  Calabi-Yau $n$-manifold $(M,\omega, J,g, \Omega)$
is a lagrangian submanifold $L\subset M$ corresponding to the
K\"{a}hler form $\omega$ such that ${\rm
Re}e^{\sqrt{-1}\theta}\Omega|_{L}=dv_{g|_{L}} $ where $dv_{g|_{L}}$
denotes the volume form of $ g|_{L}$ on $L$. Equivalently,
$\dim_{\mathbb{R}}L=n$,
$$\omega|_{L}\equiv 0, \ \ \ {\rm
Im}e^{\sqrt{-1}\theta}\Omega|_{L}\equiv 0
$$ (c.f. \cite{HL}).
 In \cite{Mc},
Mclean  showed that, for a compact special lagrangian submanifold
$L$ in a Calabi-Yau manifold $(M,\omega, J,g, \Omega)$, the local
moduli space of special lagrangian submanifolds near $L$ is a smooth
manifold of  dimension $b_{1}(L)$, and, moreover, the tangent space
of the moduli space at $L$ can be identified with  the space of
harmonic 1-forms on $(L, g|_{L})$. In \cite{Hit}, various structures
on the moduli space of special lagrangian submanifolds were studied.

  A special
lagrangian fibration on a Calabi-Yau $n$-manifold $(M,\omega,
\Omega)$ consists of a topological space $B$, and a surjection
$f:M\longrightarrow B$ such that there is an open dense subset
$B_{0}\subset B$, which is a real $n$-manifold, satisfying that, for
any $b\in B_{0}$, $f^{-1}(b)$ is a smooth special lagrangian
submanifold in $(M,\omega, \Omega)$. By \cite{Du} (see also
\cite{Gro2}), $f^{-1}(b)$, $b\in B_{0}$, is a $n$-torus.   The first
step of SYZ conjecture is to  construct  such fibration on a
Calabi-Yau manifold when  the complex structure is close to the
large complex structure limit point  enough (c.f. \cite{SYZ}). Then
the mirror manifold is a compactification  of the dual fibration of
$ f: f^{-1}(B_{0})\longrightarrow B_{0}$. Generalized
  special lagrangian fibrations were  constructed in
  some almost Calabi-Yau manifolds  in \cite{Ru}, \cite{Ru2}, \cite{Ru3},
 \cite{Go}.  In \cite{GW1}, special lagrangian fibrations were
 constructed on some Borcea-Voisin type Calabi-Yau 3-manifolds with degenerated Ricci-flat K\"{a}hler-Einstein metrics.
 In
 \cite{Ru4}, H-minimal Lagrangian fibrations, a generalization
 of special lagrangian fibration,  were  constructed on some
 regions of K\"{a}hler-Einstein manifolds with negative scalar
 curvature.

 In
    \cite{KS} and \cite{GW2}, SYZ conjecture was  refined to the
   following form: Let
    $\pi: \mathcal{M}\rightarrow \Delta  $ be  a maximally unipotent
     degeneration  of  Calabi-Yau $n$-manifolds
 over the
unit disc $\Delta\subset \mathbb{C}$, and $\alpha$ be an  ample
class on $\mathcal{M}$. For  any $t\in \Delta\backslash \{0\}$, let
$\tilde{g}_{t}$ be the unique Ricci-flat K\"{a}hler metric on
$M_{t}=\pi^{-1}(t)$ with its K\"{a}hler form $\tilde{\omega}_{t}\in
\alpha|_{M_{t}}\in H^{1,1}(M_{t}, \mathbb{R})\cap H^{2}(M_{t},
\mathbb{Z})$, and $\bar{g}_{t}={\rm
diam}_{\tilde{g}_{t}}^{-2}(M)\tilde{g}_{t}$. Then $(M_{t},
\bar{g}_{t})$
   converges to a compact  metric space $(B, d_{B})$ of Hausdorff  dimension $n$ in the Gromov-Hausdorff sense,
    when $t\rightarrow 0$. Furthermore,  there is a closed
    subset $S_{B}\subset B$ of Hausdorff  dimension $n-2$ such that
    $B\backslash S_{B}$ is an affine manifold, and $d_{B}$ is
    induced by a Monge-Amp\`{e}re metric $g_{B}$ on $B\backslash
    S_{B}$ (c.f. \cite{KS}). The mirror manifolds are supposed to be
    constructed from the dual affine structure on $B\backslash
    S_{B}$ and the metric $g_{B}$.
  This conjecture was verified    for some  K3 surfaces in
   \cite{GW2}. By using hyperK\"{a}hler rotation, some  K3 surfaces
   admit special lagrangian fibrations.    It was shown that some  K3 surfaces with Ricci-flat K\"{a}hler metrics collapse
    along such  special lagrangian fibrations in \cite{GW2}.    The two versions of SYZ conjecture suggest the
   equivalence between the  existence of special lagrangian
   submanifolds and the collapsing  of  Ricci-flat K\"{a}hler metrics on some regions of
     Calabi-Yau manifolds, when complex structures are close to the large complex limit point
     enough.

In Riemannian geometry, the  collapsing of Riemannian manifolds  was
studied by various authors  (c.f.   \cite{CG1}, \cite{CG2},
\cite{CFG}, \cite{CT}, \cite{Fu},
  and references in  \cite{Fu}), since Gromov
introduced the notion  of Gromov-Hausdorff topology in \cite{G1}. In
 \cite{CG2}, it was proved that there is a constant $\epsilon_{0}
 (n)>0$ depending only on $n$ such that there is an $F$-structure of
 positive rank on the   region   $M_{\epsilon_{0}}$ in a Riemannian
  $n$-manifold $(M,g)$, where $M_{\epsilon_{0}}$ denotes the subset
   with  injectivity radius $i_{g}(p)<  \epsilon_{0}$ and sectional curvature $
  \sup\limits_{B_{g}(p,1)}|K_{g}|\leq 1,$ for any $p\in
  M_{\epsilon_{0}}$. See \cite{CG1} and \cite{CG2} for the
  definition of $F$-structure of
 positive rank, which is a generalization of fibration. A folklore
 conjecture says that there should be special lagrangian fibrations
 on such  region in  a Calabi-Yau manifold, i.e. the  region of
 bounded curvature and sufficiently collapsed (c.f. \cite{Fu2}). The first result in the present
 paper is devoted to
 construct special lagrangian fibrations under such Riemannian
 geometric conditions.

   \vspace{0.5cm}

\begin{theorem}\label{0.1} For any $n\in \mathbb{N}$ and any $\sigma > 1 $,  there exists a constant $\epsilon=\epsilon
 (n, \sigma)>0$  depending only on $n$ and $\sigma $
 such that,  if  $(M, \omega, J, g, \Omega)$ is  a closed  Ricci-flat  Calabi-Yau
 n-manifold with   $[\omega]\in H^{2}(M, \mathbb{Z})$,  and
  $p\in M $ such that
  \begin{itemize}
   \item[i)]  the injectivity radius and
   the sectional curvature $$  i_{g}(p)<  \epsilon,  \ \ \
  \sup_{B_{g}(p,1)}|K_{g}|\leq 1 ,$$
 \item[ii)]  $[\Omega|_{B_{g}(p,\sigma i_{g}(p))}]\neq 0$ in
   $  H^{n}(B_{g}(p,\sigma i_{g}(p)),
 \mathbb{C})$,
    \end{itemize}
then   there is an open subset $W\subset M$ satisfying  that
$B_{g}(p,\sigma i_{g}(p))\subset W$, and   $(W, \omega, \Omega)$
admits a special lagrangian fibration of  a phase
$\theta\in\mathbb{R}$, i.e. there is a  topological space $B$, and a
 surjection $f:W\longrightarrow B$ such  that, for any $b\in B$, $f^{-1}(b)$ is
a smooth n-submanifold,
$$\omega|_{f^{-1}(b)}\equiv 0,  \ \ {\rm and} \ \ {\rm Im}e^{\sqrt{-1}\theta}\Omega|_{f^{-1}(b)}\equiv 0.$$
  \end{theorem}

   \vspace{0.5cm}

     \begin{remark}\label{0.3} From the proof of this theorem, we can see that $B$
     is an orbifold, and, if $b$  belongs the singular set of $B$,
     $f^{-1}(b)$  is a smooth multi-fiber.
       \end{remark}

  \begin{remark}\label{0.2} The condition ii) in the theorem  can be replaced by
   the following small non-vanishing  $n$-cycle  condition: there is an  $[A]\in
   H_{n}(B_{g}(p,\sigma i_{g}(p)),
 \mathbb{Z})$ such that $$\int_{A}\Omega \neq 0.$$ This condition
 can not be removed since it is satisfied if there is a special
 lagrangian submanifold $L$ near $p$ having comparable   size to  $i_{g}(p)$,
 for example $L\subset  B_{g}(p, \sigma i_{g}(p))$.
   \end{remark}

    \begin{remark}\label{0.21}  It is a challenging  task to verify
    condition i) in Theorem \ref{0.1}, i.e. to find the region of
    bounded curvature in a Ricci-flat Calabi-Yau manifold.  If $(M, \omega, J, g, \Omega)$ is
    a K3-surface with Ricci-flat metric, it was shown in  \cite{CT}
    that there are  universal  constants $C>0$,  $\tau >0$, and a finite subset
    $\{p_{j}\}\subset M
    $, $1\leq j \leq \tau$, such that  $$ \sup_{B_{g}(p,1)}|K_{g}|\leq C,
    $$ for any $p\in M\backslash \bigcup\limits_{1\leq j \leq
    \tau}B_{g}(p_{j},2)$. From the author's knowledge, no such
    estimate for higher dimensional Calabi-Yau manifolds is known
    except some trivial cases, for example $K3\times T^{2} $.
   \end{remark}

       Next, in the opposite direction,  we show that the existence of special
       lagrangian submanifolds  with small volume  implies the collapsing of some regions in  the
       ambient Calabi-Yau manifolds. The following  theorem is a
        corollary  of a volume  comparison theorem for calibrated
       submanifolds in \cite{Go}.

       \vspace{0.5cm}

\begin{theorem}\label{0.4}
 Let $(M, \omega, J, g, \Omega)$ be   a closed  Ricci-flat  Calabi-Yau
 n-manifold, and $p\in M$. Assume that the sectional curvature $K_{g}$
  satisfies $$ \sup_{B_{g}(p, 2 \pi)}K_{g}\leq 1,$$ and there is a special lagrangian
  submanifold $L$ of phase $\theta$ such that $p\in L$, and  $$
  \int_{L}{\rm Re} e^{\sqrt{-1}\theta}\Omega <
 \frac{\pi}{2n}\varpi_{n-1} ,$$ where $ \varpi_{n-1}$ denotes the volume of $S^{n-1}$ with the standard metric of  constant curvature 1.
   Then the
  injectivity radius $i_{g}(p) $ of $ (M,g)$  at $p$ satisfies that $$i_{g}(p)^{n}\leq \frac{n\pi^{n-1}}{2^{n-1}\varpi_{n-1}} \int_{L}{\rm Re} e^{\sqrt{-1}\theta}\Omega. $$
\end{theorem}

   \vspace{0.5cm}

 As mentioned in Remark \ref{0.21},  the assumption of bounded curvature  is not
 desirable. If we strengthen the condition of the existence of a special lagrangian
  submanifold to  the existence of a special lagrangian
   fibration,  we  can still  obtain the collapsing result in the  absence of  bounded curvature.

        \vspace{0.5cm}

     \begin{theorem}\label{0.5} For any $n\in \mathbb{N}$ and any $\varepsilon >0 $,  there is a constant $\delta = \delta(n, \varepsilon)>0$
satisfying that: Assume that  $(M, \omega, J, g, \Omega)$ is  a
closed Ricci-flat Calabi-Yau
 n-manifold, $p\in M$,  and there is a  homology class  $A\in H_{n}(M,
 \mathbb{Z})$ such that, for any $x\in B_{g}(p,1)$, there is a special
 lagrangian submanifold $L_{x}$ of phase $\theta$  passing $x$ and presenting $A$, i.e.
 $x\in L_{x}$ and $[L_{x}]=A$. If $$\int_{A}{\rm
 Re}e^{\sqrt{-1}\theta}\Omega < \delta,$$  then
 $${\Vol}_{g}(B_{g}(p,1))\leq \varepsilon .$$
\end{theorem}

   \vspace{0.5cm}

 Let  $\{(M_{k}, \omega_{k}, J_{k}, g_{k}, \Omega_{k})\}$ be  a
family of closed  Ricci-flat  Calabi-Yau
 $n$-manifolds, and $\{p_{k} \}$ be a sequence of points such that
 there are open subsets  $W_{k}\supset B_{g_{k}}(p_{k},1)$ in $M_{k}$  admitting
  special lagrangian
   fibrations  $f_{k}: W_{k} \longrightarrow B_{k}$
    of phase $\theta_{k}$ corresponding to Calabi-Yau structures $ (\omega_{k},
    \Omega_{k})$. If  $$ \lim_{k\longrightarrow\infty} \int_{f_{k}^{-1}(b_{k})}{\rm Re}
 e^{\sqrt{-1}\theta_{k}}\Omega_{k}=0, $$ where $b_{k} \in B_{k}$,  Theorem \ref{0.5} implies that $$\lim_{k\longrightarrow\infty}
 {\Vol}_{g_{k}}(B_{g_{k}}(p_{k},1)) =0,$$ and, by passing to  a subsequence, $\{(M_{k},  g_{k},
 p_{k})\}$ converges  to a  path metric space of lower Hausdorff  dimension in the pointed Gromov-Hausdorff sense  (c.f.  \cite{CT},
 \cite{An3}).   Theorem \ref{0.1},  Theorem \ref{0.4} and Theorem \ref{0.5} give an
 evidence of the  equivalence between the  existence of special lagrangian
   submanifolds and the collapsing  of  Ricci-flat K\"{a}hler metrics on
     Calabi-Yau manifolds near the large complex limit point from the  Riemannian geometry's  point of
     view.

     Theorem \ref{0.4} and Theorem
      \ref{0.5} are special cases of corresponding theorems   for Riemannian  manifolds with calibration forms and
    calibrated submanifolds (See Section 6 for
   details).   Special lagrangian
   submanifolds in Calabi-Yau manifolds  are   examples of
  calibrated submanifolds. There are some other examples of calibrated submanifolds, for example
    associative and  coassociative
  submanifolds in $G_{2}$-manifolds and Cayley
       submanifolds in $Spin(7)$-manifolds,   which have deep
       relationships  with M-theory and
       exceptional mirror symmetry (c.f.  \cite{Ac}  \cite{GYZ}).
       It is  interesting to
   understand the
    interaction  between the collapsing of  $G_{2}$-manifolds (resp. $Spin(7)$-manifolds) and
        the existence of
       associative and  coassociative submanifolds (resp.  Cayley
       submanifolds).

 The organization of the paper is as follows: In \S2, we review some notions and results, which will be
used in this  paper.   In \S3, we use the blow-up argument to give
 local  approximations  of Calabi-Yau manifolds  by
  complete flat Calabi-Yau manifolds. In \S4, we study the deformation of special lagrangian fibrations. In \S5,
we prove Theorem  \ref{0.1} by combining the results in \S3 and \S4.
Finally, we prove  Theorem \ref{0.4} and Theorem \ref{0.5}  in \S6.
\\

 \vspace{0.10cm}

 \noindent {\bf Acknowledgement:} The  author would like to  thank  Prof.
 Weidong Ruan
    and   Prof. Xiaochun Rong for useful discussions.
Thanks also goes to Prof. Fuquan Fang  for constantly support.

 \vspace{0.10cm}

  \section{Preliminaries} In this section, we  review some notions and results, which will be
used in the proof of Theorem  \ref{0.1},  Theorem \ref{0.4} and
Theorem \ref{0.5}.

  \subsection{Cheeger-Gromov convergence}
  Since Gromov introduced the concept of Gromov-Hausdorff
 convergence  in \cite{G1},  the convergence  of Riemannian manifolds was
studied from  various  perspectives   (c.f. \cite{An1}, \cite{An2},
 \cite{CT1}, \cite{CT}, \cite{Fu},
\cite{GW}, \cite{GW2}, \cite{RZ}, \cite{To} and references in
\cite{C}).   There is an extension of Gromov-Hausdorff convergence
to sequences of  pointed metric spaces for dealing with non-compact
situations.

\begin{definition}[\cite{G1}, \cite{Fu}] For two pointed  complete   metric spaces  $(X, d_{X}, x)$ and $(Y, d_{Y}, y)$, a map $\psi:
( X, x)\rightarrow (Y, y)$ is called an $\epsilon$-pointed
approximation if
  $\psi(B_{d_{X}}(x,\epsilon^{-1}))\subset
 B_{d_{Y}}(y,\epsilon^{-1})$, $\psi(x)=y$,  and $\psi|_{B_{d_{X}}(x,\epsilon^{-1})}:B_{d_{X}}(x,\epsilon^{-1})\longrightarrow
 B_{d_{Y}}(y,\epsilon^{-1})$ is an $\epsilon$-approximation,
 i.e. $B_{d_{Y}}(y,\epsilon^{-1}) \subset \{y'\in Y| d_{Y}(y', \psi(B_{d_{X}}(x,\epsilon^{-1})))<\epsilon\}$, and
 \[
 |d_{X}(x_{1}, x_{2})-d_{Y}(\psi(x_{1}), \psi(x_{2}))|<\epsilon
 \]
 for any $x_{1}$ and $ x_{2}\in B_{d_{X}}(x,\epsilon^{-1})$.  The number
 \[
 d_{GH}((X, d_{X},x),(Y,
 d_{Y},y))=\inf \left\{\epsilon\left|\begin{array}{c} {\rm There \ are} \  \epsilon-{\rm pointed \ \ \  approximations } \\ \psi:
 (X,x) \rightarrow (Y, y),  \  {\rm and} \ \phi: (Y, y) \rightarrow  (X, x)\end{array} \right.\right\}
 \]
is called pointed  Gromov-Hausdorff distance between $(X, d_{X}, x)$
and $(Y, d_{Y}, y)$.
 We say that  a family of pointed  complete   metric spaces  $(X_{k}, d_{X_{k}}, x_{k})$ converges  to a
   complete  metric space   $(Y, d_{Y},y)$ in the pointed  Gromov-Hausdorff sense, if $$\lim_{k\rightarrow\infty} d_{GH}((X_{k}, d_{X_{k}}, x_{k}),
    (Y, d_{Y},y))=0. $$
\end{definition}

The following is the  famous
  Gromov pre-compactness theorem:

   \begin{theorem}[\cite{G1}]
   \label{2002}
   Let $\{(M_{k}, g_{k},p_{k})\}$ be a
   family of pointed  complete  Riemannian
  manifolds such that Ricci curvatures ${\rm Ric}(g_{k})\geq -C $ for a constant  $C$  in-dependent of $k$. Then, a subsequence of $(M_{k},
  g_{k},p_{k})$ converges to a pointed  complete  path metric space $(Y, d_{Y}, y)$  in the pointed  Gromov-Hausdorff
  sense.
  \end{theorem}

   This  theorem  shows that a   family of
   pointed  complete Ricci-flat  Einstein  manifolds   converges to a pointed  complete  path metric
  space by passing to a subsequence in the pointed  Gromov-Hausdorff sense. The structure of the limit space was studied
  in \cite{CC1},   \cite{CT},
  \cite{CCT} and \cite{CT1} etc.  We need the following
  result in the proof of Theorem \ref{0.5}.

  \begin{theorem}[\cite{CC1}, \cite{C}]
   \label{2003}
   Let $\{(M_{k}, g_{k},p_{k})\}$ be a
   family of pointed  complete Ricci-flat  Einstein
     n-manifolds, i.e. ${\rm Ric}(g_{k})\equiv 0$,  such that  $$  {\Vol}_{g_{k}}(B_{g_{k}}(p_{k},1))\geq C, $$ for a  constant  $C$
   in-dependent of $k$, and $(Y, d_{Y}, y)$  be a pointed  complete  path metric
   space such that $$\lim_{k\rightarrow\infty} d_{GH}((M_{k}, g_{k},p_{k}),
    (Y, d_{Y},y))=0. $$ Then the Hausdorff dimension  $\dim_{\mathcal{H}}Y = n$, and  there is a closed subset $S_{Y}\subset Y$ of
    Hausdorff dimension $\dim_{\mathcal{H}}S_{Y} <  n-1$ such that $Y\backslash
    S_{Y}$ is a  n-manifold, and  $ d_{Y}$ is
    induced by a Ricci-flat  Einstein metric $g_{\infty}$ on $Y\backslash
    S_{Y}$. Furthermore, for any compact subset $D\subset Y\backslash
    S_{Y}$, there are embeddings $F_{k,D}:D\longrightarrow M_{k}$
    such that $ F_{k,D}^{*}g_{k}$ converges to $g_{\infty}$ in the
    $C^{\infty}$-sense.
\end{theorem}

 In \cite{G1} and \cite{GW}, a convergence theorem, the
Cheeger-Gromov convergence theorem,  was proved for Riemannian
manifolds with bounded curvature and non-collapsing. The K\"{a}hler
version of this theorem can be found in \cite{Ru1}. See  \cite{CT1}
for the convergence  of manifolds with other holonomy groups.

\begin{theorem}[K\"{a}hler version of Cheeger-Gromov convergence theorem] \label{2004}  Let \\ $\{(M_{k},
g_{k}, J_{k}, \omega_{k},  p_{k})\}$ be a family of pointed compact
K\"{a}hler n-manifolds with sectional curvature and injectivity
radius at $p_{k}$
$$|K_{g_{k}}|\leq 1, \ \ \ i_{g_{k}}(p_{k})\geq C, $$ for a constant
$C>0$ independent of $k$. Then a subsequence of  $\{(M_{k}, g_{k},
J_{k}, \omega_{k},  p_{k})\}$ converges to a complete  K\"{a}hler
n-manifold $(X, g, J, \omega,  p)$ in the  pointed
$C^{1,\alpha}$-sense, i.e.  for any $r>0$, there are embeddings
$F_{k,r}:B_{g}(p,r)\longrightarrow M_{k}$ such that
$F_{k,r}(p)=p_{k}$, $F_{k,r}^{*}g_{k}$ (resp.
$dF_{k,r}^{-1}J_{k}dF_{k,r}$ and $F_{k,r}^{*}\omega_{k}$) converges
to $g$ (resp. $J$ and $\omega$) in the $C^{1,\alpha}$-sense.
\end{theorem}

If we assume that $ g_{k}$  are Einstein metrics, it is shown  in
\cite{An1}  that, by passing to a subsequence, $\{(M_{k}, g_{k},
J_{k}, \omega_{k}, p_{k})\}$ converges to  $(X, g, J, \omega,  p)$
in the  pointed $C^{\infty}$-sense, and $g$ is also an Einstein
metric, i.e. $F_{k,r}^{*}g_{k}$ (resp. $dF_{k,r}^{-1}J_{k}dF_{k,r}$
and $F_{k,r}^{*}\omega_{k}$) converges to $g$ (resp. $J$ and
$\omega$) in the $C^{\infty}$-sense. Assume that $(M_{k}, g_{k},
J_{k}, \omega_{k}, p_{k})$ are Ricci-flat Calabi-Yau manifolds, and
$\Omega_{k}$ are the  corresponding holomorphic volume forms.  Since
$\Omega_{k}$ are parallel, i.e. $ \nabla^{g_{k}}\Omega_{k}\equiv 0$,
for any $r>0$, $F_{k,r}^{*}\Omega_{k}$ converge  to a  holomorphic
volume form $\Omega$ on $X$ in the $C^{\infty}$-sense, and $(X, g,
J, \omega, \Omega)$ is a complete Ricci-flat Calabi-Yau
$n$-manifold.

In \cite{CG1}, \cite{CG2}, the collapsing of Riemannian manifolds
with bounded curvature was studied by combining blow-up arguments
and the  Cheeger-Gromov convergence theorem.  It was shown that
there is a constant $\epsilon_{0}
 (n)>0$ depending only on $n$ such that there is an $F$-structure
 $\mathcal{F}$ of
 positive rank on a region covering  $M_{\epsilon_{0}}$ in a Riemannian
  $n$-manifold $(M,g)$, where $M_{\epsilon_{0}}$ denotes the subset
 with
   injectivity radius $i_{g}(p)<  \epsilon_{0}$ and sectional curvature $
  \sup\limits_{B_{g}(p,1)}|K_{g}|\leq 1,$ for any $p\in
  M_{\epsilon_{0}}$. See \cite{CG1} and \cite{CG2} for the
  definition of $F$-structure of
 positive rank. If we assume that $g$ is a K\"{a}hler metric, some
 additional information about the  $F$-structure
 $\mathcal{F}$ is expected. We have the following conjecture:

 \begin{conjecture} For any $n\in \mathbb{N}$,  there exists  a constant $\epsilon=\epsilon
 (n)>0$ depending only on $n$
 such that,  if  $(M, \omega, J, g)$ is  a closed K\"{a}hler
 n-manifold with   $[\omega]\in H^{2}(M, \mathbb{Z})$,  and
  $$M_{\epsilon}=\{p\in M | \  i_{g}(p)<  \epsilon, \
  \sup_{B_{g}(p,1)}|K_{g}|\leq 1 \},$$ then there is an open subset $W\subset M$
  such that $W\supset M_{\epsilon}$, and $W$  admits an F-structure $\mathcal{F}$
of positive rank, whose orbits $\mathcal{O}_{p} $, $p\in
M_{\epsilon}$, are isotropic submanifolds of $(M, \omega)$, i.e.
$$\omega|_{\mathcal{O}_{p}}\equiv 0.$$
 \end{conjecture}

We will address this question  in other papers. In the present
paper, we prove Theorem  \ref{0.1} by combining Theorem \ref{2004}
and the deformation theory of  special lagrangian fibrations.

\subsection{A  comparison theorem  for calibrated submanifolds}
In  \cite{HL}, Harvey and Lawson introduced the notion of calibrated
submanifold. If  $(M,  g)$ is   a
    Riemannian  manifold, and  $\Theta $ is  a closed   $n$-form such that $\Theta|_{\xi}\leq dv_{\xi} $ for any oriented $n$-plane
     $\xi$  in the tangent bundle of $M$,    then   $\Theta
    $ is called a calibration on $M$,  where  $dv_{\xi} $ denotes the
    volume form on $ \xi$.  An oriented $n$-submanifold $L$ of $M$
    is called  calibrated
   by  the calibration  $\Theta$, if $\Theta|_{L}
  $ equals   to the volume form of $g|_{L}$ on $L$.  Mclean studied the deformation theory of calibrated  submanifolds in
\cite{Mc}.

There are some examples of calibrated submanifolds:
 holomorphic submanifolds in K\"{a}hler  manifolds,  special
lagrangian
  submanifolds in Calabi-Yau manifolds, associative coassociative
  submanifolds in $G_{2}$-manifolds, and Cayley submanifolds in
  $Spin(7)$-manifolds (c.f. \cite{HL} \cite{J1}) etc.
  If $(M, \omega,
  J,g)$  is a  K\"{a}hler   $m$-manifold, then $\frac{1}{n!}\omega^{n} $, $n\leq m$, are calibrations on $M$, and
  holomorphic $n$-submanifolds are calibrated
   by $\frac{1}{n!}\omega^{n} $.
    If $(M,\omega, J,g, \Omega)$ is
  a Ricci-flat  Calabi-Yau $n$-manifold, then, for any $\theta\in \mathbb{R} $,  ${\rm Re}e^{\sqrt{-1}\theta}\Omega
  $ is a calibration on $M$, and a special lagrangian submanifold
  $L$ of phase $\theta$ is  calibrated
   by  ${\rm Re}e^{\sqrt{-1}\theta}\Omega
  $.  If  $(M, g)$ is a Riemannian  manifold with holonomy group
  $G_{2}$, then $M$ admits a parallel 3-form $ \phi$, which is a
  calibration on $M$, and $*_{g}\phi$ is a calibration 4-form on $M$.  Submanifolds calibrated by $\phi$ are called
  associative submanifolds, and submanifolds calibrated by $*_{g}\phi$ are called
  coassociative submanifolds (c.f. \cite{J1}). If  $(M, g)$ is a Riemannian  manifold with holonomy group
  $Spin(7)$, then $M$ admits a calibration 4-form $ \Omega$, and Cayley
  submanifolds are submanifolds calibrated by $\Omega$ (c.f. \cite{J1}).

    In \cite{Go}, a volume  comparison theorem  for calibrated
submanifolds  was obtained.

  \begin{theorem}[Theorem 2.0.1.   in \cite{Go}]\label{6.04} Let $(M,  g)$ be   a closed  Riemannian  manifold, $\Theta $ be a calibration  n-form,
    and $p\in M$. Assume that the sectional curvature $K_{g}$
  satisfies $$ \sup_{B_{g}(p, \frac{2\pi}{\sqrt{\Lambda}})}K_{g}\leq \Lambda, \ \ \ \ \Lambda>0, $$ and there is a
  submanifold $L$ calibrated by   $\Theta$ such that    $p\in
  L$.  Then  $${\Vol}_{g}(B_{g}(p,r)\cap L)\geq {\Vol}_{h_{1}}(B_{h_{1}}(r)),
  $$ for any $r\leq \min \{i_{g}(p), \frac{\pi}{\sqrt{\Lambda}}\}$, where  $h_{1}$ denotes the standard  metric on
  $S^{n} $ with constant curvature $\Lambda$ , and $B_{h_{1}}(r)$ denotes a
  metric $r$-ball in  $S^{n} $.
  \end{theorem}

\subsection{Implicit function theorem}
For studying the deformation  of  special lagrangian fibrations,  we
need the following quantity version of implicit function theorem.

\begin{theorem}[Theorem 3.2 in \cite{Ru}]\label{2.1}
   Let $ (\mathfrak{B}_{1}, \|\cdot\|_{1})$ and $(\mathfrak{B}_{2},
\|\cdot\|_{2})$ be two Banach spaces, $\|\cdot\|_{E}$ be the
standard Euclidean metric on $\mathbb{R}^{n} $, $U\subset
\mathbb{R}^{n}\times \mathfrak{B}_{1}$ be an open set, and
$\mathfrak{F}: U\longrightarrow \mathfrak{B}_{2}$ be a continuously
differentiable map. Denote the differential
$$D\mathfrak{F}(y,\sigma)(\dot{y}+\dot{\sigma})=
D_{y}\mathfrak{F}(y,\sigma)\dot{y}+D_{\sigma}\mathfrak{F}(y,\sigma)\dot{\sigma},$$
for $(y,\sigma)\in U$, $\dot{y}\in\mathbb{R}^{n}$ and
$\dot{\sigma}\in \mathfrak{B}_{1}$. Assume that  $(0, 0)\in U$
satisfies that $D_{\sigma}\mathfrak{F}(0, 0):
\mathfrak{B}_{1}\longrightarrow \mathfrak{B}_{2}$ has a bounded
linear inverse $ D_{\sigma}\mathfrak{F}(0, 0)^{-1}:
\mathfrak{B}_{2}\longrightarrow \mathfrak{B}_{1}$ with $$
\|D_{\sigma}\mathfrak{F}(0, 0)^{-1}\|\leq \overline{C}$$ for a
constant $\overline{C}>0$. Let $r>0$,  $\delta_{0}>\delta>0$ be
constants  such that, if $\|y_{0}\|_{E}< r,$ and $\|\sigma\|_{1}\leq
\delta_{0}$, then $(y_{0}, \sigma)\in U $, and
$$ \|D_{\sigma}\mathfrak{F}(y_{0}, \sigma)-D_{\sigma}\mathfrak{F}(0, 0)\|\leq
\frac{1}{2\overline{C}} \ \ {\rm and} \ \ \|\mathfrak{F}(y_{0},
0)\|_{2}\leq \frac{\delta}{4\overline{C}}.$$ Then, for any $
\|y\|_{E}< r$, there exists a unique $\sigma(y)\in \mathfrak{B}_{1}$
such that
$$\mathfrak{F}(y, \sigma(y))=0, \ \ \
\|\sigma(y)\|_{1}\leq \delta.$$ Furthermore,
$$D\sigma(y)\dot{y}=-D_{\sigma}\mathfrak{F}(y, \sigma)^{-1}D_{y}\mathfrak{F}(y,
\sigma)\dot{y}.$$
\end{theorem}

The difference between this version of implicit function theorem and
the usual  one (c.f. \cite{GT}) is that we use the condition
$\|\mathfrak{F}(y, 0)\|_{2}\leq \frac{\delta}{4\overline{C}}$ to
replace the condition $\mathfrak{F}(0, 0)=0 $ besides other quantity
estimates.

 \vspace{0.10cm}

    \section{The blow-up limit }
    Let $\{(M_{k}, \omega_{k}, J_{k}, g_{k}, \Omega_{k})\}$ be   a family of  closed  Ricci-flat  Calabi-Yau
 $n$-manifold with   $[\omega_{k}]\in H^{2}(M_{k}, \mathbb{Z})$,  and
  $p_{k}\in M_{k} $. Assume  that
  \begin{itemize}
   \item[i)]  the injectivity radius and
   the sectional curvature $$  i_{g_{k}}(p_{k})<  \frac{1}{k},  \ \ \
  \sup_{B_{g_{k}}(p_{k},1)}|K_{g_{k}}|\leq 1 ,$$
 \item[ii)] there is a $ \sigma \gg 1$ such that   $[\Omega_{k}|_{B_{g_{k}}(p_{k},\sigma i_{g_{k}}(p_{k}))}]\neq 0$ in
  $  H^{n}(B_{g_{k}}(x_{k},\sigma i_{g_{k}}(p_{k})),
 \mathbb{C})$.
    \end{itemize}
 If we denote  $\tilde{\omega}_{k}=i^{-2}_{g_{k}}(p_{k})\omega_{k}$,
   $\tilde{g}_{k}=i^{-2}_{g_{k}}(p_{k})g_{k}$, and
   $\tilde{\Omega}_{k}=i^{-n}_{g_{k}}(p_{k})\Omega_{k}$, then $$  i_{\tilde{g}_{k}}(p_{k})=1,  \ \ \
  \sup_{B_{\tilde{g}_{k}}(p_{k},k)}|K_{\tilde{g}_{k}}|\leq \frac{1}{k^{2}}
  ,$$ and $[\tilde{\Omega}_{k}|_{B_{\tilde{g}_{k}}(p_{k},\sigma )}]\neq 0$ in
  $  H^{n}(B_{\tilde{g}_{k}}(p_{k},\sigma ),
 \mathbb{C})$. By the  Cheeger-Gromov's  convergence  theorem (c.f. Theorem \ref{2004}),
   a subsequence of $(M_{k}, \tilde{\omega}_{k}, \tilde{g}_{k}, J_{k},  \tilde{\Omega}_{k},  p_{k})
  $ converges to a complete flat Calabi-Yau  $n$-manifold   $(X,
  \omega_{0},
  g_{0}, J_{0}, \Omega_{0},  p_{0})$ in the $C^{\infty}$-sense,
  i.e. for any $r>\sigma$, there are embeddings $F_{r,k}:
 B_{g_{0}}(p_{0}, r)\longrightarrow M_{k}$ such that $F_{r,k}(p_{0})=p_{k}
 $, and
  $F_{r,k}^{*}\tilde{g}_{k}$ (resp. $F_{r,k}^{*}\tilde{\omega}_{k}$ and $F_{r,k}^{*}\tilde{\Omega}_{k}$)
  converges to $ g_{0}$ (resp. $ \omega_{0}$ and $\Omega_{0}$) in the
  $C^{\infty}$-sense.
The purpose of this section is to prove that $(X,
  \omega_{0},
  \Omega_{0})$ admits a special lagrangian fibration.

 By the smooth convergence,   $i_{g_{0}}(p_{0})=\lim\limits_{k\rightarrow \infty}i_{\tilde{g}_{k}}(p_{k})=1$. The  soul theorem (c.f. \cite{CG2}, \cite{Pe}) implies that    there is
  a
compact  flat totally geodesic   submanifold $S\subset X$, the soul,
such that $(X, g_{0})$ is isometric to the total space of the normal
bundle $\nu (S)$ with a metric induced by $g_{0}|_{S}$ and a natural
flat connection.

   \begin{lemma}\label{3.01}  $\dim_{\mathbb{R}}S \geq  n$.
    \end{lemma}

 \begin{proof} If $\dim_{\mathbb{R}}S <  n$,   then $$H^{n}(X, \mathbb{C})= H^{n}(T_{r}(S), \mathbb{C})=
 H^{n}(S, \mathbb{C})=\{0\}$$ for any $r>0$, where $T_{r}(S)=\{p\in X| dist_{g_{0}}(p, S)\leq r\}
 $. Let $r_{0}>r_{1}> \sigma$ such that $T_{r_{1}}(S)\subset B_{g_{0}}(x_{0}, r_{0})
 $, and $ F_{r_{0},k}(T_{r_{1}}(S))\supset B_{\tilde{g}_{k}}(x_{k},\sigma
 )$ for $k\gg 1$.  Then the inclusion maps induce homeomorphisms  on cohomology
 groups $$H^{n}(M,\mathbb{C})\longrightarrow H^{n}(F_{r_{0},k}(T_{r_{1}}(S)),\mathbb{C})\longrightarrow H^{n}(B_{\tilde{g}_{k}}(x_{k},\sigma
 ),\mathbb{C}),  $$ $$ {\rm and} \ \ \ [\tilde{\Omega}_{k}]\mapsto [\tilde{\Omega}_{k}|_{F_{r_{0},k}(T_{r_{1}}(S))}]
 \mapsto [\tilde{\Omega}_{k}|_{B_{\tilde{g}_{k}}(x_{k},\sigma
 )}]\neq 0 . $$ Thus $ [\tilde{\Omega}_{k}|_{F_{r_{0},k}(T_{r_{1}}(S))}]\neq
 0$ in $ H^{n}(F_{r_{0},k}(T_{r_{1}}(S)),\mathbb{C})$, which
 contradicts to $$ H^{n}(F_{r_{0},k}(T_{r_{1}}(S)),\mathbb{C})\cong H^{n}(T_{r_{1}}(S),
 \mathbb{C})=\{0\}.$$
  \end{proof}

  If $\pi_{h}: \tilde{S}\longrightarrow S$ is  the holonomy
  covering of $S$, Bieberbach's theorem (c.f. \cite{CG2}, \cite{Pe})  says that $(\tilde{S},\pi_{h}^{*}g_{0})$ is isometric to  a
  flat torus, and $\pi_{h}$ has finite  order at most $\lambda(n)$,
  for a constant $\lambda(n)$ depending only on $n$.
  If we denote   $\bar{\pi}: \mathbb{C}^{n}\longrightarrow X
 $ the universal covering of $X$ with $\bar{\pi}(0)\in S$, then  $\bar{S}=\bar{\pi}^{-1}(S)$ is a real linear subspace of $\mathbb{C}^{n}$,  and $
 \omega_{E}=\bar{\pi}^{*}\omega_{0}$ (resp.
 $\Omega_{E}=\bar{\pi}^{*}\Omega_{0}$) is the standard flat  K\"{a}hler
 form (resp. the  standard holomorphic volume form), i.e.  $\omega_{E}=
 \sqrt{-1}\sum_{\alpha}dz_{\alpha}\wedge d\bar{z}_{\alpha}$ and $ \Omega_{E}=dz_{1}\wedge \cdots \wedge dz_{n} $ under some coordinates $z_{1},
  \cdots, z_{n} $ on $\mathbb{C}^{n}$.
   Note that there is a lattice $\Lambda \subset \bar{S}$ such that
  $\tilde{S}= \bar{S}/\Lambda$.  If we denote $\mathfrak{q}:
  \bar{S}\longrightarrow \tilde{S} $ the quotient map, then
  $\bar{\pi}=\pi_{h}\circ \mathfrak{q}$.

   \begin{lemma}\label{3.02} $\dim_{\mathbb{R}}\bar{S}=n$,  and  there is a  constant
   $\theta_{0}\in\mathbb{R}$ such that  $\omega_{E}|_{\bar{S}}=0$
   and ${\rm
    Im}e^{\sqrt{-1}\theta_{0}}\Omega_{E}|_{\bar{S}}=0.  $
 Moreover,
    $S$ is a special  lagrangian  submanifold  of  phase $\theta_{0}$ in  $(X,
    \omega_{0},  \Omega_{0})$, i.e. $\dim_{\mathbb{R}}S=n$,  $$ \omega_{0}|_{S}\equiv 0, \ \ {\rm and } \  \ {\rm
    Im}e^{\sqrt{-1}\theta_{0}}\Omega_{0}|_{S}=0.  $$
  \end{lemma}

 \begin{proof}
If $\omega_{E}|_{\bar{S}}\neq 0$, and  thus  $\omega_{0}|_{S}\neq
0$, then  there are two vectors $v_{1}, v_{2}\in \bar{S}$ such that
$\omega_{E}(v_{1}, v_{2})>0 $. By perturbing $v_{1}$ and $ v_{2}$ a
little bit if necessary, we have that  $\tilde{\Sigma}=
\mathfrak{q}(\{t_{1}v_{1}+t_{2}v_{2}| t_{i}\in \mathbb{R} \})$
 is a closed  2-torus  in $\tilde{S}$, i.e. a closed 2-parameters subgroup. Thus $\Sigma=\pi_{h}(\tilde{\Sigma})$ is a closed oriented  surface in $S$,  which satisfies
$$\int_{\Sigma}\omega_{0}\geq \frac{1}{\lambda(n)}\int_{\tilde{\Sigma}}\pi_{h}^{*}\omega_{0}\geq
 \frac{\omega_{E}(v_{1}, v_{2})}{\lambda(n)\|v_{1}\wedge  v_{2}\|_{h_{E}}}V_{E}>0,$$ where $V_{E}$ denotes the Euclidean area
 of the intersection of
 $ \{t_{1}v_{1}+t_{2}v_{2}| t_{i}\in \mathbb{R} \}$ with  the fundamental domain of the quotient map $\mathfrak{q}$.  From
  the smooth  convergence of $(M_{k}, \tilde{\omega}_{k}, \tilde{g}_{k})$,  $$
   \lim_{k\longrightarrow \infty} i^{-2}_{g_{k}}(p_{k})\int_{F_{r,k}(\Sigma)}\omega_{k} =\lim_{k\longrightarrow \infty}
  \int_{F_{r,k}(\Sigma)}\tilde{\omega}_{k}=\lim_{k\longrightarrow \infty}
  \int_{\Sigma}F_{r,k}^{*}\tilde{\omega}_{k}=\int_{\Sigma}\omega_{0},$$ for $r\gg 1$ such that $ \Sigma\subset
  B_{g_{0}}(p_{0},r)$.
  Thus $$ 0< \frac{1}{2}i^{2}_{g_{k}}(p_{k})\int_{\Sigma}\omega_{0} \leq
  \int_{F_{r,k}(\Sigma)}\omega_{k}\leq
  2i^{2}_{g_{k}}(p_{k})\int_{\Sigma}\omega_{0}\leq 2k^{-2}\int_{\Sigma}\omega_{0}< 1,$$  for $k\gg 1$.
  Since $[F_{r,k}(\Sigma)]\in H_{2}(M_{k}, \mathbb{Z})$ and $[\omega_{k}]\in H^{2}(M_{k}, \mathbb{Z})$,   we obtain $$\int_{F_{r,k}(\Sigma)}\omega_{k}\in \mathbb{Z}, $$ which is  a
  contradiction. Hence $\omega_{E}|_{\bar{S}}\equiv 0$ and  $  \omega_{0}|_{S}\equiv 0$,  which implies
  that $S$ is a lagrangian submanifold $(X,
    \omega_{0})$ by  combining  Lemma \ref{3.01}.

  Since  $\bar{S}$ is a lagrangian linear  subspace  of $(\mathbb{C}^{n}, \omega_{E})
    $,  there is a $\theta_{0}\in \mathbb{R}$ such that $ {\rm
    Im}e^{\sqrt{-1}\theta_{0}}\Omega_{E}|_{\bar{S}}=0$.
    This implies that $\bar{S}$ is a special  lagrangian linear  subspace
       of  phase $\theta_{0}$ in  $(\mathbb{C}^{n}, \omega_{E}, \Omega_{E})
    $. Thus $S$ is a special  lagrangian
    submanifold of  phase $\theta_{0}$ in  $(X,
    \omega_{0}, \Omega_{0})$, i.e. $$\omega_{0}|_{S}=0, \ \ \  {\rm
    Im}e^{\sqrt{-1}\theta_{0}}\Omega_{0}|_{S}=0. $$
  \end{proof}

  \begin{lemma}\label{3.03} For $k\gg 1$, $[F_{r,k}^{*}\tilde{\omega}_{k}|_{S}]=0
  $ in $H^{2}(S, \mathbb{R})$.
  \end{lemma}

   \begin{proof} By the smooth convergence of $\tilde{\omega}_{k} $ and Lemma
   \ref{3.02},  $$
   \lim_{k\longrightarrow\infty}\int_{F_{r,k}(A)}\tilde{\omega}_{k}=\int_{A}\omega_{0}=0,$$
   for any cycle $A\in H_{2}(S, \mathbb{Z})$. For $k\gg 1$, we have
   $$
   |\int_{F_{r,k}(A)}\omega_{k}|=i^{2}_{g_{k}}(p_{k})|\int_{F_{r,k}(A)}\tilde{\omega}_{k}|<\frac{1}{2k^{2}}<1.$$
   Since $[F_{r,k}(A)]\in H_{2}(M_{k}, \mathbb{Z})$ and $[\omega_{k}]\in H^{2}(M_{k}, \mathbb{Z})
   $, we obtain $\int_{F_{r,k}(A)}\omega_{k}\in \mathbb{Z}$. This implies that   $|\int_{F_{r,k}(A)}\omega_{k}|=0$, and we obtain the
  conclusion $$
  \int_{A}F_{r,k}^{*}\tilde{\omega}_{k}=0.  $$
    \end{proof}

    Let $\tilde{X}$ be the total space of the pull-back  $\pi_{h}^{*}\nu(S)$ of the normal
    bundle. Note that we can identify the zero section of
    $\pi_{h}^{*}\nu(S)$ with $\tilde{S} $, and the covering $\pi_{h} $  extends to a finite covering  $\pi:\tilde{X}\longrightarrow
 X$ of $X$, i.e.
   $\tilde{S}=\pi^{-1}(S)\subset \tilde{X}$, and
 $\pi|_{\tilde{S}}= \pi_{h}$.   The fundamental group  $\pi_{1}(\tilde{S})\cong \pi_{1}(\tilde{X})
 $ is isomorphic to the lattice $\Lambda$, $\pi_{1}(\tilde{X})$ is a normal subgroup
 of $ \pi_{1}(X) =\pi_{1}(S) $, and the covering
  group
  $\Gamma \cong \pi_{1}(S)/\pi_{1}(\tilde{S})=\pi_{1}(X)/\pi_{1}(\tilde{X})$.
  Note that  $\pi_{1}(\tilde{X})$ (resp. $\pi_{1}(X)$) acts on $\mathbb{C}^{n}
  $ preserving $g_{E}$,  $\omega_{E}$ and
  $\Omega_{E}$,  $\bar{S}$ is invariant, $\tilde{X}=\mathbb{C}^{n}/\pi_{1}(\tilde{X})
  $ (resp. $X=\mathbb{C}^{n}/\pi_{1}(X)
  $), and $\tilde{S}=\bar{S}/\pi_{1}(\tilde{S})=\bar{S}/\Lambda
  $ (resp. $S=\bar{S}/\pi_{1}(S)
  $).

    \begin{proposition}\label{3.1}
 Let $\bar{S}^{\perp}$ be the orthogonal  complement of
  $\bar{S}$ in $\mathbb{C}^{n}$, i.e. $\mathbb{C}^{n}=\bar{S}\oplus \bar{S}^{\perp}$,
  and $g_{E}(v, w)=0$, for any $v\in \bar{S}$ and $w\in
  \bar{S}^{\perp}$. Then
     \begin{itemize}
   \item[i)]  $(\tilde{X}, \pi^{*}g_{0})$ is isometric to $(T^{n}\times \bar{S}^{\perp}, h+h_{E}) $, where $T^{n}=\bar{S}/\Lambda=\tilde{S}$,
    $h_{E}=g_{E}|_{\bar{S}^{\perp}}$, and  $h$ is the standard flat metric on $T^{n}$ induced by $ g_{E}|_{\bar{S}}$.
    \item[ii)] The action of $\Gamma$ on $ \tilde{X}$ is a product
    action, i.e. there are $\Gamma$-actions on $T^{n}$ and $
    \bar{S}^{\perp}$ such that $\gamma \cdot (x,y)=(\gamma \cdot x, \gamma \cdot
    y)$ for any $\gamma\in \Gamma $,  $x\in T^{n}$ and $ y\in \bar{S}^{\perp}$.
    Furthermore,
      $T^{n}\times \{0\}$ is $\Gamma$-invariant,
     and $S=\pi(T^{n}\times \{0\}) = (T^{n}\times \{0\})/\Gamma $.
      \item[iii)]  $$\pi^{*}\omega_{0}|_{T^{n}\times \{y\}}\equiv 0,  \ \ {\rm
      and } \ \
       \pi^{*}{\rm Im}e^{\sqrt{-1}\theta_{0}}\Omega_{0}|_{T^{n}\times \{y\}}\equiv 0, $$ for any $y\in \bar{S}^{\perp}$, and a constant
       $\theta_{0} \in \mathbb{R}$.  \end{itemize}
  \end{proposition}

   \begin{proof}We choose coordinates $x_{1}, \cdots, x_{n}$ on $\bar{S}$
  and $y_{1}, \cdots, y_{n}$ on $\bar{S}^{\perp}$ such that
  $$g_{E}=\sum(dx_{j}^{2}+dy_{j}^{2}), \ \ \ \omega_{E}=\sum
  dx_{j}\wedge dy_{j}, \ \ \ e^{\sqrt{-1}\theta_{0}}\Omega_{E}=\bigwedge_{j=1}^{n}(dx_{j}+\sqrt{-1}dy_{j}). $$

 If $\mathcal {G}$ is a subgroup of  the fundamental group $\pi_{1}(X)=\pi_{1}(S)  $, then $\mathcal {G}$  acts on
  $\mathbb{C}^{n}$ preserving $g_{E}$,  $\omega_{E}$ and
  $\Omega_{E}$, and $\bar{S}$ is a  invariant subspace. For
  any $\gamma\in \mathcal {G} $, we have  $\gamma \cdot
  (v+w)=G_{\gamma}(v+w)+b_{\gamma}$, where $G_{\gamma}\in U(\mathbb{C}^{n})$,   $b_{\gamma}\in
  \bar{S}$, $v\in \bar{S}$ and $w\in
  \bar{S}^{\perp}$.  Since $\bar{S}$ is  invariant, we obtain then
  $G_{\gamma}(v+w)=A_{\gamma}v+B_{\gamma}w+ C_{\gamma}w$ where $A_{\gamma}\in  SO(\bar{S})$,
  $B_{\gamma}\in
  SO(\bar{S}^{\perp})$, and $C_{\gamma}\in Hom
  (\bar{S}^{\perp},\bar{S})$. Moreover, $G_{\gamma}\in
  SO(\mathbb{R}^{2n})$ implies $C_{\gamma}=0 $.  Since
  $\omega_{E}(G_{\gamma}(v+w), G_{\gamma}(v+w))=\omega_{E}(v+w,v+w)$,
  we have
  $B_{\gamma}=A_{\gamma}^{-1,T}=A_{\gamma}$, and $\gamma \cdot (v+w)=A_{\gamma}(v+w)+b_{\gamma}$.
    Thus $\pi_{1}(\tilde{X})\cong \Lambda$ acts on
 $\mathbb{C}^{n}$ given by $\gamma \cdot (v+w)=v+w+b_{\gamma}$,
 $b_{\gamma}\in \Lambda$,
   for any $v\in \bar{S}$ and $w\in
  \bar{S}^{\perp}$. This implies that
  $\tilde{X}=\mathbb{C}^{n}/\pi_{1}(\tilde{X})\cong
  \bar{S}/\Lambda \times \bar{S}^{\perp}=\tilde{S} \times \bar{S}^{\perp}$,
  and $\pi^{*}g_{0}=h+h_{E}$ where $h_{E}=g_{E}|_{\bar{S}^{\perp}}$, and  $h$ is the standard flat metric on $\tilde{S}
  $ induced by $ g_{E}|_{\bar{S}}$.

 The $\pi_{1}(X)$-action on $ \mathbb{C}^{n}$ descents to a
 $\Gamma$-action on $\tilde{X}$, which is a product action since the
 $\pi_{1}(X)$-action is so. Moreover, $\tilde{S} \times \{0\} $ is a
 invariant set as $\bar{S} \times \{0\} $ is invariant under the
 $\pi_{1}(X)$-action. If we denote the quotient map $\mathfrak{q}_{1}:
  \mathbb{C}^{n}\longrightarrow \mathbb{C}^{n}/\Lambda=\tilde{X}$,
  then $\bar{\pi}=\pi\circ \mathfrak{q}_{1}$,
  $g_{E}=\mathfrak{q}_{1}^{*}\pi^{*}g_{0}$,
  $\omega_{E}=\mathfrak{q}_{1}^{*}\pi^{*}\omega_{0}$, and
  $\Omega_{E}=\mathfrak{q}_{1}^{*}\pi^{*}\Omega_{0}$. Since
  $\omega_{E}|_{\bar{S}\times\{y\}}=0$ and
  $e^{\sqrt{-1}\theta_{0}}\Omega_{E}|_{\bar{S}\times\{y\}}=0$ for
  $y\in \bar{S}^{\perp}$, we obtain that $$\pi^{*}\omega_{0}|_{T^{n}\times \{y\}}\equiv 0, \ \ {\rm
  and }  \ \ \pi^{*}{\rm Im}e^{\sqrt{-1}\theta_{0}}\Omega_{0}|_{T^{n}\times \{y\}}\equiv 0 , $$  for  a constant
       $\theta_{0} \in \mathbb{R}$.
   \end{proof}

\begin{remark}\label{3.2} The coordinates $x_{1}, \cdots, x_{n}$ on
$\bar{S}$ in the proof of this proposition induce parallel 1-forms
$dx_{1}, \cdots, dx_{n}$
 on $(\tilde{S}, h)$£¬ which are pointwise linear
independent, i.e. $dx_{1}, \cdots, dx_{n}$ is a global  parallel
frame field.  Under the coordinates $y_{1}, \cdots, y_{n}$ on
$\bar{S}^{\perp}$, we have these
  formulas
  $$\pi^{*}g_{0}=\sum(dx_{j}^{2}+dy_{j}^{2}), \ \ \ \pi^{*}\omega_{0}=\sum
  dx_{j}\wedge dy_{j}, \ \ \ e^{\sqrt{-1}\theta_{0}}\pi^{*}\Omega_{0}=\bigwedge_{j=1}^{n}(dx_{j}+\sqrt{-1}dy_{j}).
  $$
   \end{remark}

\begin{remark}\label{3.3}
  The natural     projection
  $f_{0}:\tilde{X}\longrightarrow \bar{S}^{\perp}$ is equivariant
  under the $\Gamma$ actions on $\tilde{X}$ and $ \bar{S}^{\perp}$. For any $y\in  \bar{S}^{\perp}$,
   $f_{0}^{-1}(y)=\tilde{S}\times \{y\}$, and  $f_{0}$ is a special lagrangian
  fibration  on $(\tilde{X}, \pi^{*}\omega_{0}, e^{\sqrt{-1}\theta_{0}}\pi^{*}\Omega_{0})
  $, i.e. $\dim_{ \mathbb{R}}f_{0}^{-1}(y)=n$,  $$\pi^{*}\omega_{0}|_{f_{0}^{-1}(y)}\equiv 0, \
  \ \  e^{\sqrt{-1}\theta_{0}}\pi^{*}\Omega_{0}|_{f_{0}^{-1}(y)}\equiv 0.  $$
  \end{remark}

\vspace{0.10cm}

\section{Local special lagrangian fibrations} In this section, we
study the deformation of special lagrangian fibrations under the
convergence of Calabi-Yau metrics.  Let $(Y,
  \omega,
  g, J, \Omega)$ be a complete flat  Calabi-Yau  $n$-manifold.
  \begin{condition}\label{c4.001} Assume that
   \begin{itemize}
   \item[i)]   $Y=T^{n}\times
  \mathbb{R}^{n}$,    $g=h+h_{E}$,  and the natural
   projection $f: Y\longrightarrow \mathbb{R}^{n}$ is a special lagrangian
  fibration of $(Y,
  \omega,  \Omega)$,  where $T^{n}=\mathbb{R}^{n}/\Lambda$ is a torus,
   $\Lambda$ is a lattice in $ \mathbb{R}^{n}$, $h_{E}$ is the standard Euclidean metric on $ \mathbb{R}^{n}$, and $h$
   is the standard flat metric induced by $h_{E}$.
    \item[ii)] We
  assume that  there are  parallel 1-forms
$dx_{1}, \cdots, dx_{n}$  on $(T^{n},h)$, which are pointwise linear
independent, and coordinates $y_{1},
  \cdots, y_{n}$ on $ \mathbb{R}^{n}$ such that $$g=h+h_{E}=\sum(dx_{j}^{2}+dy_{j}^{2}), \ \ \ \omega=\sum
  dx_{j}\wedge dy_{j}, \ \ \ \Omega=\bigwedge_{j=1}^{n}(dx_{j}+\sqrt{-1}dy_{j}).
  $$ \item[iii)]There is a family of Calabi-Yau structures $(\omega_{k}, g_{k},
J_{k}, \Omega_{k}) $   converging  to $(\omega,
  g, J, \Omega)$ in the $C^{\infty}$-sense on
  $Y_{2r}=T^{n}\times B_{h_{E}}(0,2r)$ for a $r\gg 1$, where $B_{h_{E}}(0,2r)=\{y\in \mathbb{R}^{n}|
   \|y\|_{h_{E}}<2r\}$. Moreover, $\omega_{k}\in [\omega]\in H^{2}(Y_{2r}, \mathbb{R})$.    \item[vi)] There is a finite
  group $\Gamma$ acting on $Y_{2r}$ preserving $\omega_{k}, g_{k},
  \Omega_{k},\omega,  g,  \Omega$, and  $T^{n}\times \{0\} $ is a invariant set. The $\Gamma$-action
    is a product action on $T^{n}\times B_{h_{E}}(0,2r)$. The natural
   projection $f: Y\longrightarrow \mathbb{R}^{n}$ is
   $\Gamma$-equivariant.
   \end{itemize}
   \end{condition}
 The goal of this section is to  construct equivariant   special lagrangian
  fibrations  on $(Y_{r}, \omega_{k},  \Omega_{k}) $ for $k\gg 1$.

Denote $L=T^{n}\times \{0\}$, which is a special  lagrangian
submanifold of $(Y, \omega, \Omega)$, i.e. $\omega|_{L}=0$ and ${\rm
Im} \Omega|_{L}=0$. Note that we can identify $Y$ with the total
space of the normal bundle $\nu(L)$ by  the exponential map from
$\nu(L)$ to $Y$,  $ \exp_{L,g}: (x,
\sum_{j}y_{j}\frac{\partial}{\partial y_{j}})\mapsto (x,y)$ where
$x\in L $ and $y=(y_{1},\cdots,y_{n}) $. There is a canonical bundle
isomorphism from $\nu(L)$ to the cotangent bundle  $T^{*}L$ given by
$v\mapsto \iota (v)\omega$ where $v\in \nu_{x}(L)$. Thus we can
identify $Y$ with the total space of  $T^{*}L$ by the map
\begin{equation}\label{e4.001}(x,y)\mapsto (x, \iota(\sum_{j}y_{j}\frac{\partial}{\partial
y_{j}})\omega )=(x, \sum_{j} y_{j}dx_{j}), \end{equation} where
$x\in L $ and $y=(y_{1},\cdots,y_{n})\in \mathbb{R}^{n} $. We do not
distinguish $Y$ with   $T^{*}L$ in this section for convenience. For
a  1-form  $\sigma$  on
 $L$, and a $y\in \mathbb{R}^{n}$, which can be regarded as a 1-form
  from  above,   $$L(y, \sigma)=\{(x, y+\sigma(x))| x\in L\} $$ denotes  the  graph of
 $y+\sigma$,
  i.e.  $y=\sum
y_{j}dx_{j}$, $\sigma=\sum\sigma_{j}dx_{j}$, and
 $$L(y, \sigma)=\{(x, y_{1}+\sigma_{1}(x),
\cdots, y_{n}+\sigma_{n}(x))| x\in L\}.$$

There are two constants $a_{k}>0$ and $\theta_{k}\in \mathbb{R}$,
for any $k$,
 such that $$
 \int_{L}\Omega_{k}=a_{k}e^{-\sqrt{-1}\theta_{k}}\int_{L}\Omega_{0}=
 a_{k}e^{-\sqrt{-1}\theta_{k}}\int_{L}{\rm Re}\Omega_{0},$$  $\lim\limits_{k\longrightarrow\infty} a_{k}=1$ and $\lim\limits_{k\longrightarrow\infty}
 \theta_{k}=0$ by the smooth convergence of $ \Omega_{k}$. There are real 1-forms
$\alpha_{k}$ and complex value  $(n-1)$-forms $\beta_{k}$ such that
$$\omega_{k}=\omega_{0}-d\alpha_{k},
 \ \ \ \Omega_{k}=a_{k}e^{-\sqrt{-1}\theta_{k}}(\Omega_{0}+d\beta_{k})$$ by
$\omega_{k}\in [\omega]\in H^{2}(Y_{2r}, \mathbb{R})$.
 By the smooth convergence of $ \omega_{k}$ and $ \Omega_{k}$, \begin{equation}\label{e4.010}
 \lim_{k\longrightarrow \infty}\|d\alpha_{k}\|_{C^{2}(Y_{2r},g)}= \lim_{k\longrightarrow
 \infty}\|d\beta_{k}\|_{C^{2}(Y_{2r},g)}=0.
 \end{equation}

Define a diffeomorphism  $\Pi: L\longrightarrow L(y, \sigma)$ by
$x\mapsto (x, y+\sigma(x))$ for a $y\in \mathbb{R}^{n}$ and a 1-form
$\sigma$ on $L$. If
\begin{equation}\label{e4.01} \mathfrak{F}_{k}(y,\sigma)=(-\Pi^{*}\omega_{k}|_{L(y,
\sigma)},*_{h}a_{k}^{-1}\Pi^{*}{\rm
Im}e^{\sqrt{-1}\theta_{k}}\Omega_{k}|_{L(y, \sigma)}),\end{equation}
where $ *_{h}$ is the Hodge star operator on $(L,h)$,  then $L(y,
\sigma)$ is a special lagrangian submanifold of $(Y, \omega_{k},
\Omega_{k})$ of  phase $\theta_{k} $ if and only if $$
\mathfrak{F}_{k}(y,\sigma)=0.$$ A straightforward calculation  (c.f.
\cite{Mc}) gives
\begin{equation}\label{e4.02} \mathfrak{F}_{k}(y,\sigma)=(d\sigma
+\Pi^{*}d \alpha_{k}|_{L(y, \sigma)},
*_{h}d*_{h}\sigma
+*_{h}\Pi^{*}d{\rm Im}\beta_{k}|_{L(y, \sigma)}).\end{equation}

 We denote $\Omega^{j}(L)$ the space of $j$-forms on $L$, and  define two Banach spaces
    $\mathfrak{B}_{1}=
C^{1,\alpha}(d\Omega^{0}(L)\oplus d^{*_{h}}\Omega^{2}(L))$ and $
\mathfrak{B}_{2}= C^{0,\alpha}(d\Omega^{1}(L)\oplus
d^{*_{h}}\Omega^{1}(L))$. Then  $\mathfrak{F}_{k}$ defines  a smooth
map $\mathfrak{F}_{k}: \mathcal{U}(r) \longrightarrow
\mathfrak{B}_{2}$ for any $k$, where
$\mathcal{U}(r)=\{\|y\|_{h_{E}}+\|\sigma\|_{C^{1,\alpha}(L,h)}<2r|
(y,\sigma)\in \mathbb{R}^{n}\times \mathfrak{B}_{1}\}$.

\begin{lemma}\label{4.01} For any $y\in B_{h_{E}}(0,2r)$,  $$\|\mathfrak{F}_{k}(y,0)\|_{C^{0,\alpha}(L,h)}\leq
 C \|(d\alpha_{k},d\beta_{k})\|_{C^{1,\alpha}(Y_{2r},g)},$$ for a
 constant $C$ independent of $k$.
 \end{lemma}

\begin{proof} Since  $$
\mathfrak{F}_{k}(y,0)=(\Pi^{*}d \alpha_{k}|_{L(y,0)},
*_{h}\Pi^{*}d{\rm Im}\beta_{k}|_{L(y,0)}),$$ we obtain the conclusion by
straightforward calculations.
\end{proof}
The differentials of $\mathfrak{F}_{k}(y,\sigma)$  are
\begin{equation}\label{e4.03}D_{\sigma}\mathfrak{F}_{k}(y,\sigma)\dot{\sigma}=(d\dot{\sigma},*_{h}d*_{h}\dot{\sigma}
) +(D_{\sigma}(\Pi^{*}d \alpha_{k}|_{L(y, \sigma)})\dot{\sigma},
*_{h}D_{\sigma}(\Pi^{*}d{\rm Im}\beta_{k}|_{L(y, \sigma)})\dot{\sigma}),
\end{equation}
\begin{equation}\label{e4.04}D_{y}\mathfrak{F}_{k}(y,\sigma)\dot{y}=(D_{y}(\Pi^{*}d
\alpha_{k}|_{L(y, \sigma)})\dot{y},
*_{h}D_{y}(\Pi^{*}d{\rm Im}\beta_{k}|_{L(y, \sigma)})\dot{y})
\end{equation} $${\rm and} \ \ \
D\mathfrak{F}_{k}(y,\sigma)(\dot{y}+\dot{\sigma})=D_{\sigma}\mathfrak{F}_{k}(y,\sigma)\dot{\sigma}+D_{y}\mathfrak{F}_{k}(y,\sigma)\dot{y}.$$
 Under the frame field  $dx_{1},\cdots, dx_{n}$ and coordinates
$y_{1},\cdots, y_{n}$,
  $$d \alpha_{k}=\sum_{ij}( \alpha_{k,ij}dx_{i}\wedge dx_{j}+ \alpha_{k,i(n+j)}
dx_{i}\wedge dy_{j} +\alpha_{k,(n+i)(n+j)}dy_{i}\wedge dy_{j}).$$
The differential is
\begin{eqnarray*}D(\Pi^{*}d \alpha_{k}|_{L(y,
\sigma)})(\dot{y}+\dot{\sigma})& =& \sum_{ijl} (\frac{\partial
\alpha_{k,ij}}{\partial
y_{l}}(\dot{y}_{l}+\dot{\sigma}_{l})dx_{i}\wedge dx_{j}+
\alpha_{k,i(n+j)} dx_{i}\wedge d\dot{\sigma}_{j}\\ & \ & +
\frac{\partial \alpha_{k,i(n+j)}}{\partial
y_{l}}(\dot{y}_{l}+\dot{\sigma}_{l})dx_{i}\wedge
d\sigma_{j}+\alpha_{k,(n+i)(n+j)}d\sigma_{i}\wedge d\dot{\sigma}_{j}
\\ & \ & \frac{\partial \alpha_{k,(n+i)(n+j)}}{\partial
y_{l}}(\dot{y}_{l}+\dot{\sigma}_{l})d\sigma_{i}\wedge d\sigma_{j} ).
\end{eqnarray*} We obtain  \begin{equation}\label{e4.05}
\|D_{\sigma}(d
\alpha_{k}|_{L})\dot{\sigma}\|_{C^{0,\alpha}(L,h)}\leq C\|d
\alpha_{k}\|_{C^{1,\alpha}(Y_{2r},g)}
\|\dot{\sigma}\|_{C^{1,\alpha}(L,h)},
\end{equation}
$$\|D_{\sigma}(\Pi^{*}d \alpha_{k}|_{L(y,
\sigma)})\dot{\sigma}\|_{C^{0,\alpha}(L,h)}\leq C\|d
\alpha_{k}\|_{C^{1,\alpha}(Y_{2r},g)}(\sum_{l=0,1,2}\|\sigma\|_{C^{1,\alpha}(L,h)}^{l})
\|\dot{\sigma}\|_{C^{1,\alpha}(L,h)},
$$ $$\|D_{y}(\Pi^{*}d \alpha_{k}|_{L(y,
\sigma)})\dot{y}\|_{C^{0,\alpha}(L,h)}\leq C\|d
\alpha_{k}\|_{C^{1,\alpha}(Y_{2r},g)}(\sum_{l=0,1,2}\|\sigma\|_{C^{1,\alpha}(L,h)}^{l})
\|\dot{y}\|_{h_{E}},
$$ for a
 constant $C$ independent of $k$.  The same argument gives  \begin{equation}\label{e4.06}
\|D_{\sigma}(d {\rm
Im}\beta_{k}|_{L})\dot{\sigma}\|_{C^{0,\alpha}(L,h)} \leq C\|d
\beta_{k}\|_{C^{1,\alpha}(Y_{2r},g)}
\|\dot{\sigma}\|_{C^{1,\alpha}(L,h)},
\end{equation}
$$\|D_{\sigma}(\Pi^{*}d {\rm
Im}\beta_{k}|_{L(y, \sigma)})\dot{\sigma}\|_{C^{0,\alpha}(L,h)}\leq
C  \|d \beta_{k}\|_{C^{1,\alpha}(Y_{2r},g)}(\sum_{l=0,1,\cdots,
n}\|\sigma\|_{C^{1,\alpha}(L,h)}^{l})
\|\dot{\sigma}\|_{C^{1,\alpha}(L,h)},
$$ $$\|D_{y}(\Pi^{*}d {\rm
Im}  \beta_{k}|_{L(y, \sigma)})\dot{y}\|_{C^{0,\alpha}(L,h)}\leq C
\|d \beta_{k}\|_{C^{1,\alpha}(Y_{2r},g)}(\sum_{l=0,1,\cdots,
n}\|\sigma\|_{C^{1,\alpha}(L,h)}^{l}) \|\dot{y}\|_{h_{E}}.
$$

 \begin{lemma}\label{4.03} The operator $D_{\sigma}\mathfrak{F}_{k}(0,0)$ is  invertible  for  $k\gg 1$, and   $$ \|D_{\sigma}\mathfrak{F}_{k}(0,0)^{-1}\|\leq
 \overline{C},$$ for a constant $\overline{C}>0$ independent of $k$.
  \end{lemma}

  \begin{proof} Note that $$D_{\sigma}\mathfrak{F}_{k}(0,0)\dot{\sigma}=
  (d\dot{\sigma},*_{h}d*_{h}\dot{\sigma}
)+(D_{\sigma}(d \alpha_{k}|_{L})\dot{\sigma},
*_{h}D_{\sigma}(d{\rm
Im}\beta_{k}|_{L})\dot{\sigma})=(\mathcal{D}+V_{k})\dot{\sigma},
$$  where  $\mathcal{D}=d-*_{h}d*_{h}$ is the restriction of  the Hodge Dirac operator $d+d^{*_{h}} $ on the space of
 1-forms, and, thus,  is an   elliptic operator of
1-order.  By the standard elliptic estimate (c.f. Proposition
 1.5.2 in \cite{J1} and \cite{GT}), we have
$$ \|\xi\|_{C^{1,\alpha}(L,h)}\leq C_{S}\|\mathcal{D}\xi\|_{C^{0,\alpha}(L,h)},
$$ for any $\xi\in \mathfrak{B}_{1}$, and  a
 constant $C_{S}$ independent of $k$. Hence  $\mathcal{D}$ is
injective. From the definition of $\mathfrak{B}_{2}$, $\mathcal{D}$
is also  surjective, which implies that  $\mathcal{D}$  is
invertible from $\mathfrak{B}_{1}$ to $\mathfrak{B}_{2}$.
 Moreover,
$$\|\mathcal{D}^{-1}\|\leq C_{S}.  $$ By (\ref{e4.05}) and (\ref{e4.06}),
  $$\|V_{k}\|\leq C \|(d
\alpha_{k}, d \beta_{k})\|_{C^{1,\alpha}(Y_{2r},g)}<
\frac{1}{2C_{S}},$$ for $k\gg 1$,
 and, thus,
$$\|\mathcal{D}^{-1}V_{k}\|<\frac{1}{2}.
$$ By the standard operator's theory (c.f. \cite{Sa}),
$D_{\sigma}\mathfrak{F}_{k}(0,0)=\mathcal{D}+V_{k}$ is invertible,
and the inverse operator  is defined by
$$
D_{\sigma}\mathfrak{F}_{k}(0,0)^{-1}=(\sum_{j=0}^{\infty}(-1)^{j}(\mathcal{D}^{-1}V_{k})^{j})\mathcal{D}^{-1}.$$
 We obtain  $$ \|D_{\sigma}\mathfrak{F}_{k}(0,0)^{-1}\|\leq
(\sum_{j=0}^{\infty}2^{-j})\|\mathcal{D}^{-1}\|\leq \overline{C}, $$
for a constant $\overline{C}>0$ independent of $k$.
  \end{proof}

  \begin{lemma}\label{4.04} For any  $\delta_{0} \ll 1$, there is a
  constant
  $k_{0}\gg 1$ such that, if $ \|y\|_{h_{E}}\leq \frac{3r}{2}$ and $\|\sigma\|_{C^{1,\alpha}(L,h)}\leq
  \delta_{0}$, and $k>k_{0}$,  then $$
\|D_{\sigma}\mathfrak{F}_{k}(y,\sigma)-D_{\sigma}\mathfrak{F}_{k}(0,0)\|\leq
\frac{1}{2\overline{C}}.$$  Furthermore,
$D_{\sigma}\mathfrak{F}_{k}(y,\sigma)$ is also invertible, and  $$
\|D_{\sigma}\mathfrak{F}_{k}(y,\sigma)^{-1}\|\leq 2\overline{C}.$$
  \end{lemma}
 \begin{proof}
By (\ref{e4.03}),
\begin{eqnarray*}(D_{\sigma}\mathfrak{F}_{k}(y,\sigma)-D_{\sigma}\mathfrak{F}_{k}(0,0))\dot{\sigma} & = &(
(D_{\sigma}(\Pi^{*}d \alpha_{k}|_{L(y, \sigma)})-D_{\sigma}(d
\alpha_{k}|_{L}))\dot{\sigma}, \\ & &
*_{h}(D_{\sigma}(\Pi^{*}d{\rm
Im} \beta_{k}|_{L(y, \sigma)})-D_{\sigma}(d {\rm
Im}\beta_{k}|_{L}))\dot{\sigma}).\end{eqnarray*} We can take  a
$k_{0}\gg 1$ such that, for $k>k_{0}$,
\begin{eqnarray*}
\|D_{\sigma}\mathfrak{F}_{k}(y,\sigma)-D_{\sigma}\mathfrak{F}_{k}(0,0)\|
& \leq & 2C\|(d \alpha_{k}, d
\beta_{k})\|_{C^{1,\alpha}(Y_{2r},g)}(\sum_{l=0,1,\cdots,
n}\|\sigma\|_{C^{1,\alpha}(L,h)}^{l})\\ & \leq & 2C\|(d \alpha_{k},
d \beta_{k})\|_{C^{1,\alpha}(Y_{2r},g)}n\delta_{0} \\ & \leq &
\frac{1}{4\overline{C}},
\end{eqnarray*} by  (\ref{e4.010}),  (\ref{e4.05}) and (\ref{e4.06}).  We obtain  the first formula in the conclusion.

Note that $D_{\sigma}\mathfrak{F}_{k}(y,\sigma)=
D_{\sigma}\mathfrak{F}_{k}(0,0)+(D_{\sigma}\mathfrak{F}_{k}(y,\sigma)-D_{\sigma}\mathfrak{F}_{k}(0,0))$,
$D_{\sigma}\mathfrak{F}_{k}(0,0)$ is invertible, and
$\|D_{\sigma}\mathfrak{F}_{k}(0,0)^{-1}\|\leq \overline{C}$. By the
same arguments as in the proof of Lemma \ref{4.03}, and
$$\|D_{\sigma}\mathfrak{F}_{k}(0,0)^{-1}(D_{\sigma}\mathfrak{F}_{k}(y,\sigma)-D_{\sigma}\mathfrak{F}_{k}(0,0))\|\leq \frac{1}{2}
, $$ $D_{\sigma}\mathfrak{F}_{k}(y,\sigma)$ is also invertible, and
$$\|D_{\sigma}\mathfrak{F}_{k}(y,\sigma)^{-1}\|\leq
(\sum_{j=0}^{\infty}2^{-j})\|D_{\sigma}\mathfrak{F}_{k}(0,0)^{-1}\|\leq
2 \overline{C}.$$
 \end{proof}

   \begin{lemma}\label{4.05}  For a fixed $\delta<\delta_{0}  $, there is a $k_{1}>k_{0}$ such that,  for any $y\in B_{h_{E}}(0,
 \frac{3r}{2})$ and  $k>k_{1}$,  there is a unique $\sigma_{k}(y)\in
 \mathfrak{B}_{1}$,   such that $$ \mathfrak{F}_{k}(y, \sigma_{k}(y))=0, \ \
 \ \|\sigma_{k}(y)\|_{C^{1,\alpha}(L,h)}\leq \delta, $$  which
 implies that $L(y, \sigma_{k}(y))$ is a special lagrangian submanifold
 of $(Y_{2r}, \omega_{k}, \Omega_{k})$. Furthermore,  $$\|D\sigma_{k}(y)\|\leq 2n\delta
   \overline{C}C  \|(d \alpha_{k}, d
\beta_{k})\|_{C^{1,\alpha}(Y_{2r},g)},$$ for a constant $C$
independent of $k$.
     \end{lemma}

\begin{proof}Fix a  $\delta<\delta_{0}  $, there is a $k_{1}>k_{0}$ such that, for
 $k>k_{1}$, and  any $y\in B_{h_{E}}(0, 2r)$,  $$\|\mathfrak{F}_{k}(y,0)\|_{C^{0,\alpha}(L,h)}\leq
 \frac{\delta}{4\overline{C}},$$
  by Lemma \ref{4.01}. By  Theorem
  \ref{2.1},
 Lemma  \ref{4.03} and \ref{4.04}, for any $y\in B_{h_{E}}(0,
 \frac{3r}{2})$ and $k>k_{1}$, there is a unique $\sigma_{k}(y)\in
 \mathfrak{B}_{1}$   such that \begin{equation}\label{e4.060} \mathfrak{F}_{k}(y, \sigma_{k}(y))=0, \ \
 \ \|\sigma_{k}(y)\|_{C^{1,\alpha}(L,h)}\leq \delta, \end{equation} which
 implies that $L(y, \sigma_{k}(y))$ is a special lagrangian submanifold
 of $(Y_{2r}, \omega_{k}, \Omega_{k})$.

  By (\ref{e4.04}) (\ref{e4.05}) and
(\ref{e4.06}),
 \begin{eqnarray*}\|D_{y}\mathfrak{F}_{k}(y,\sigma_{k})\|& \leq & C
\|(d \alpha_{k},d
\beta_{k})\|_{C^{1,\alpha}(Y_{2r},g)}(\sum_{l=0,1,\cdots,
n}\|\sigma_{k}\|_{C^{1,\alpha}(L,h)}^{l})\\ & \leq & C \|(d
\alpha_{k},d
\beta_{k})\|_{C^{1,\alpha}(Y_{2r},g)}n\delta,\end{eqnarray*} for a
constant $C$ independent of $k$.  By Theorem \ref{2.1},
 $$D\sigma_{k}(y)\dot{y}=-D_{\sigma}\mathfrak{F}_{k}(y,\sigma_{k})^{-1}D_{y}\mathfrak{F}_{k}(y,\sigma_{k})\dot{y}.$$
 We obtain the conclusion from Lemma \ref{4.04}.
\end{proof}

\begin{proposition}\label{4.1} For $k\gg 1$, there is an open set
  $Y_{2r}\supset W_{k}\supset Y_{r}$ such that $(W_{k}, \omega_{k},
  \Omega_{k})$ admits a equivariant  special lagrangian fibration
  $f_{k}:W_{k}\longrightarrow B_{k}$ of  phase $\theta_{k}$ over $B_{k}\subset \mathbb{R}^{n}$,  i.e. there is a $\Gamma$-action on $B_{k}$,
   $f_{k}$ is a $\Gamma$-equivariant  map, and $f_{k}$ is a special lagrangian fibration of  phase
   $\theta_{k}$,  i.e. $$ \omega_{k}|_{f_{k}^{-1}(b)}\equiv 0, \ \ \  {\rm Im}e^{\sqrt{-1}\theta_{k}}\Omega_{k}|_{f_{k}^{-1}(b)}\equiv
   0,$$ for any $b\in B_{k}$.
  \end{proposition}

\begin{proof} By Lemma \ref{4.05},  there is a unique
   $C^{1}$-map $$\sigma_{k}: B_{h_{E}}(0, \frac{3r}{2}
)\longrightarrow C^{1,\alpha}(d\Omega^{0}(T^{n})\oplus
d^{*_{h}}\Omega^{2}(T^{n})), \ \ {\rm  by } \ \ y\mapsto
\sigma_{k}(y),$$ which satisfies  $$ \mathfrak{F}_{k}(y,
\sigma_{k}(y))=0, \ \
 \ \|\sigma_{k}(y)\|_{C^{1,\alpha}(L,h)}\leq \delta\ll 1,$$ $$ {\rm and} \  \ \ \|D\sigma_{k}(y)\|\leq 2n\delta
   \overline{C}C  \|(d \alpha_{k}, d
\beta_{k})\|_{C^{1,\alpha}(Y_{2r},g)}.$$ This implies  $$
|\frac{\partial \sigma_{k,j}(y)}{\partial y_{i}}| \leq 2n\delta
   \overline{C}C  \|(d \alpha_{k}, d
\beta_{k})\|_{C^{1,\alpha}(Y_{2r},g)}\ll 1,$$  for $k\gg k_{1}>1$.

 Define a map $\Psi_{k}:
 Y_{\frac{3}{2}r}\longrightarrow Y_{2r}$ by $$\Psi_{k}:
 (x,y)\mapsto (x,y_{1}+\sigma_{k,1}(y), \cdots,
 y_{n}+\sigma_{k,n}(y))=(x,y+\sigma_{k}(y)).$$ Note that the frame
 field $ dx_{1}, \cdots,
   dx_{n}$ induces local coordinates $ x_{1}, \cdots,
   x_{n}$  around  any point on $L$, and
  the differential can be  expressed as   $$d\Psi_{k}:
 (\dot{x},\dot{y})\mapsto (\dot{x}_{j}+\sum \frac{\partial \sigma_{k,j}(y)}{\partial x_{i}}\dot{x}_{i},
  \dot{y}_{j}+\sum \frac{\partial \sigma_{k,j}(y)}{\partial
  y_{i}}\dot{y}_{i})$$ under such   local coordinates.  Thus $d\Psi_{k}$  is an
 isomorphism when $k\gg 1$, which implies that  $\Psi_{k}$ is an immersion.
 Furthermore,     for
 $y_{1}\neq y_{2}\in \mathbb{R}^{n}$,  \begin{eqnarray*}\Psi_{k}(x,y_{2})- \Psi_{k}(x,y_{1})& = & (x, \cdots, \int_{0}^{1}
 (1+\frac{\partial \sigma_{k,j}((1-t)y_{2}+ty_{1})}{\partial y_{j}}dt)( y_{2,j}-y_{1,j}), \cdots )\\  & \neq  & 0. \end{eqnarray*}
 Hence
 $\Psi_{k}$ is an embedding.

Note that the  $\Gamma$-action on $Y_{2r}=T^{n}\times
B_{h_{E}}(0,2r)$ preserves  $\omega_{k}, g_{k},
  \Omega_{k},\omega,  g,  \Omega$, and is a product action on $T^{n}\times
  B_{h_{E}}(0,2r)$, i.e. there are $\Gamma$-actions on $T^{n}$ and
  $B_{h_{E}}(0,2r)$ such that $\gamma\cdot (x,y)= (\gamma\cdot x,\gamma\cdot y)
  $ for any $\gamma\in \Gamma$, $x\in T^{n}$, and $y\in
  B_{h_{E}}(0,2r)$. Under the  identification map (\ref{e4.001}),  \begin{eqnarray*}(\gamma\cdot x,
  \gamma\cdot y)=  (\gamma\cdot x, \iota(\gamma_{*}\sum_{j}y_{j}\frac{\partial}{\partial
y_{j}})\omega )& = & (\gamma\cdot x, \gamma^{*}\omega
(\sum_{j}y_{j}\frac{\partial}{\partial y_{j}}, \gamma^{-1}_{*}
\cdot))\\ & = & (\gamma\cdot x, \gamma^{-1,*}\sum_{j} y_{j}dx_{j}).
\end{eqnarray*}
Thus   \begin{eqnarray*} \gamma \cdot L(y, \sigma_{k}(y))& = &
\{(\gamma \cdot x , \gamma \cdot (y_{1}+\sigma_{k,1}(y)(x), \cdots,
y_{n}+\sigma_{k,n}(y)(x)))| x\in T^{n}\}\\ &  = & \{(\gamma \cdot x
, \gamma^{-1,*}\sum_{j} (y_{j}+\sigma_{k,j}(y)(x))dx_{j}| x\in
T^{n}\} \\ &= & L(\gamma \cdot y,
  \gamma^{-1,*}\sigma_{k}(y)),  \end{eqnarray*} for any $\gamma\in \Gamma$  and $y\in
  B_{h_{E}}(0,2r)$. Since the  $\Gamma$-action  preserves
  $\omega_{k}$ and $
  \Omega_{k}$, $L(\gamma \cdot y,
  \gamma^{-1,*}\sigma_{k}(y))$ are special lagrangian
  submanifolds.  By the uniqueness of $\sigma_{k}(y)$,
  $\gamma^{-1,*}\sigma_{k}(y)=\sigma_{k}(\gamma \cdot
  y)\in C^{1,\alpha}(d\Omega^{0}(T^{n})\oplus
d^{*_{h}}\Omega^{2}(T^{n}))$.  Hence \begin{eqnarray*}\Psi_{k}
 (\gamma\cdot x, \gamma\cdot y) & = & (\gamma\cdot x,  \gamma\cdot y +\sigma_{k}(\gamma\cdot y))\\ &
 = &
 (\gamma\cdot x,  \gamma^{-1,*} \sum_{j}(y_{j}+\sigma_{k,j}( y))dx_{j})\\ & =&
 (\gamma\cdot x,  \gamma  \cdot (y_{1}+\sigma_{k,1}( y), \cdots, y_{n}+\sigma_{k,n}( y)))\\ & = &
 \gamma\cdot \Psi_{k}
 ( x,  y),   \end{eqnarray*}  i.e. $\Psi_{k}$ is a $\Gamma$-equivariant
 map.

  We denote $\mathcal{P}: Y_{2r}\longrightarrow B_{h_{E}}(0, 2r)$
  the natural projection, $B_{k}= B_{h_{E}}(0,
  \frac{3}{2}r)$ and  $W_{k} = \Psi_{k}(Y_{\frac{3}{2}r})$. Since the $\Gamma$-action on $B_{h_{E}}(0, 2r)$ preserves the metric $h_{E}
  $ and $0$, $B_{k}= B_{h_{E}}(0,
  \frac{3}{2}r)$ is invariant. By $\delta \ll  1\ll  r$, $W_{k} \supset Y_{r}$.    Then  $f_{k}=\mathcal{P}\circ
  \Psi_{k}^{-1}: W_{k}\longrightarrow B_{k}$ is a $\Gamma$-equivariant  special lagrangian fibration of $( W_{k}, \omega_{k}, \Omega_{k})
  $ of phase $ \theta_{k}$. We obtain the conclusion.
\end{proof}

    \section{Proof of Theorem  \ref{0.1}}
    Now we are ready to prove Theorem  \ref{0.1}.
     \begin{proof}[Proof of Theorem  \ref{0.1}] Assume that  the conclusion  is not
     true.  Then, for any fixed $\sigma >1 $,  there is a  family of  closed  Ricci-flat  Calabi-Yau
 $n$-manifolds   $\{(M_{k}, \omega_{k}, J_{k}, g_{k}, \Omega_{k})\}$   with   $[\omega_{k}]\in H^{2}(M_{k}, \mathbb{Z})$,  and
  $p_{k}\in M_{k} $ such  that
  \begin{itemize}
   \item[i)]  the injectivity radius and
   the sectional curvature $$  i_{g_{k}}(p_{k})<  \frac{1}{k},  \ \ \
  \sup_{B_{g_{k}}(p_{k},1)}|K_{g_{k}}|\leq 1 ,$$
 \item[ii)]    $[\Omega_{k}|_{B_{g_{k}}(p_{k},\sigma i_{g_{k}}(p_{k}))}]\neq 0$ in
  $  H^{n}(B_{g_{k}}(p_{k},\sigma i_{g_{k}}(p_{k})),
 \mathbb{C})$.
 \item[iii)] for any open subset $W_{k}'\supset B_{g_{k}}(p_{k},\sigma
 i_{g_{k}}(p_{k}))$, $(W_{k}', \omega_{k}, \Omega_{k})$ wouldn't
 admit special lagrangian fibrations.
    \end{itemize}

   If we denote  $\tilde{\omega}_{k}=i^{-2}_{g_{k}}(p_{k})\omega_{k}$,
   $\tilde{g}_{k}=i^{-2}_{g_{k}}(p_{k})g_{k}$, and
   $\tilde{\Omega}_{k}=i^{-n}_{g_{k}}(p_{k})\Omega_{k}$, then $$  i_{\tilde{g}_{k}}(p_{k})=1,  \ \ \
  \sup_{B_{\tilde{g}_{k}}(p_{k},k)}|K_{\tilde{g}_{k}}|\leq \frac{1}{k^{2}}
  ,$$ and $[\tilde{\Omega}_{k}|_{B_{\tilde{g}_{k}}(p_{k},\sigma )}]\neq 0$ in
  $  H^{n}(B_{\tilde{g}_{k}}(p_{k},\sigma ),
 \mathbb{C})$. By the  Cheeger-Gromov's  convergence  theorem (c.f. Theorem
 \ref{2004}),
   a subsequence of $(M_{k}, \tilde{\omega}_{k}, \tilde{g}_{k}, J_{k},  \tilde{\Omega}_{k},  p_{k})
  $ converges to a complete flat Calabi-Yau  $n$-manifold   $(X,
  \omega_{0},
  g_{0}, J_{0}, \Omega_{0},  p_{0})$ in the $C^{\infty}$-sense,
  i.e. for any $r>\sigma$, there are embeddings $F_{r,k}:
 B_{g_{0}}(p_{0}, r)\longrightarrow M_{k}$ such that $F_{r,k}(p_{0})=p_{k}
 $, and
  $F_{r,k}^{*}\tilde{g}_{k}$ (resp. $F_{r,k}^{*}\tilde{\omega}_{k}$ and $F_{r,k}^{*}\tilde{\Omega}_{k}$)
  converges to $ g_{0}$ (resp. $ \omega_{0}$ and $\Omega_{0}$) in the
  $C^{\infty}$-sense. Furthermore,   $i_{g_{0}}(p_{0})=1$. The  soul theorem (c.f. \cite{CG2}, \cite{Pe}) implies that    there is
  a
compact  flat totally geodesic   submanifold $S\subset X$, the soul,
such that $(X, g_{0})$ is isometric to the total space of the normal
bundle $\nu (S)$ with a metric induced by $g_{0}|_{S}$ and a natural
flat connection.

By Proposition (\ref{3.1}), there is a finite normal  covering  $\pi
: \tilde{X}\longrightarrow X$ with covering group $\Gamma $ such
that
     \begin{itemize}
   \item[i)]  $(\tilde{X}, \pi^{*}g_{0})$ is isometric to $(T^{n}\times \mathbb{R}^{n}, h+h_{E}) $, where $T^{n}=\mathbb{R}^{n}/\Lambda $,
    $ \Lambda$ is a lattice  in $ \mathbb{R}^{n} $,  $h_{E}$ is the standard Euclidean metric on $\mathbb{R}^{n}$,
    and  $h$ is the standard flat metric on $T^{n}$ induced by $ h_{E}$.
    \item[ii)] The action of $\Gamma$ on $ \tilde{X}$ is a product
    action, i.e. there are $\Gamma$-actions on $T^{n}$ and $
   \mathbb{R}^{n}$ such that $\gamma \cdot (x,y)=(\gamma \cdot x, \gamma \cdot
    y)$ for any $x\in T^{n}$ and $ y\in \mathbb{R}^{n}$.
    Furthermore,
      $T^{n}\times \{0\}$ is $\Gamma$-invariant, and  $S=(T^{n}\times
      \{0\})/\Gamma$.
      \item[iii)]  $$\pi^{*}\omega_{0}|_{T^{n}\times \{y\}}\equiv 0,  \ \ {\rm
      and } \ \
       \pi^{*}{\rm Im}e^{\sqrt{-1}\theta_{0}}\Omega_{0}|_{T^{n}\times \{y\}}\equiv 0, $$  for any $y\in \mathbb{R}^{n} $, and a constant
       $\theta_{0} \in \mathbb{R}$.  \end{itemize} Note that  the
       $\Gamma$-action on $\mathbb{R}^{n} $ preserves $h_{E}$, and,
       $B_{h_{E}}(0, \rho)$, which implies that  $\tilde{X}_{\rho}=T^{n}\times B_{h_{E}}(0, \rho)$ are
       invariant, for any $\rho >0$.
    Lemma (\ref{3.03}) shows   $[F_{r,k}^{*}\tilde{\omega}_{k}|_{S}]=0
  $ in $H^{2}(S, \mathbb{R})$, for $k\gg 1$, which implies $[\pi^{*}F_{r,k}^{*}\tilde{\omega}_{k}|_{\tilde{X}_{\rho}}]=0
  $ in $H^{2}(\tilde{X}_{\rho}, \mathbb{R})$, for any $\rho >0$.   By Remark (\ref{3.2}), there are  parallel 1-forms
$dx_{1}, \cdots, dx_{n}$  on $(T^{n},h)$, which are pointwise linear
independent, and coordinates $y_{1},
  \cdots, y_{n}$ on $ \mathbb{R}^{n}$ such that
  $$\pi^{*}g_{0}=\sum(dx_{j}^{2}+dy_{j}^{2}), \ \ \ \pi^{*}\omega_{0}=\sum
  dx_{j}\wedge dy_{j}, \ \ \ e^{\sqrt{-1}\theta_{0}}\pi^{*}\Omega_{0}=\bigwedge_{j=1}^{n}(dx_{j}+\sqrt{-1}dy_{j}).
  $$ Hence Condition \ref{c4.001} is satisfied.

Let $r>\rho \gg \sigma$ such that $ B_{g_{0}}(p_{0}, \sigma)\subset
\pi (\tilde{X}_{\rho})\subset \pi (\tilde{X}_{2\rho})\subset
B_{g_{0}}(p_{0}, r)$.
 By Proposition   (\ref{4.1}), for $k\gg 1$, there is an open set
  $W_{k}\supset \tilde{X}_{\rho}$ such that $(W_{k}, \pi^{*}F_{r,k}^{*}\omega_{k},
  \pi^{*}F_{r,k}^{*}\Omega_{k})$ admits a equivariant  special lagrangian fibration
  $f_{k}:W_{k}\longrightarrow B_{k}$ of  phase $\theta_{k}$, where $B_{k}\subset \mathbb{R}^{n}$,
    i.e. there is a $\Gamma$-action on $B_{k}$,
   $f_{k}$ is a $\Gamma$-equivariant  map, and  $$\pi^{*}F_{r,k}^{*} \omega_{k}|_{f_{k}^{-1}(b)}\equiv 0, \ \ \  \pi^{*}F_{r,k}^{*}
   {\rm Im}e^{\sqrt{-1}\theta_{k}}\Omega_{k}|_{f_{k}^{-1}(b)}\equiv
   0,$$ for any $b\in B_{k}$. Hence $f_{k}$ induces a special lagrangian
   fibration $\bar{f}_{k}:\pi(W_{k})\longrightarrow B_{k}/\Gamma$,
   which implies that $(F_{r,k} \circ \pi(W_{k}), \omega_{k},
   \Omega_{k})$ admits a special lagrangian
   fibration, and $F_{r,k} \circ \pi(W_{k}) \supset B_{g_{k}}(x_{k},\sigma
 i_{g_{k}}(x_{k}))$. It is a contradiction. We obtain the
 conclusion.
      \end{proof}

      \vspace{0.10cm}

  \section{Estimates for injectivity
radius}
 In this section, we prove  Theorem \ref{0.4} (Corollary \ref{6.4}) and Theorem \ref{0.5} (Theorem \ref{6.6}).
  The following estimate for
injectivity radius is a direct consequence of  Theorem \ref{6.04}.

\begin{corollary}\label{6.4}
 Let $(M,  g)$ be   a closed  Riemannian  manifold, $\Theta $ be a calibration  n-form,  and $p\in M$. Assume that the sectional curvature $K_{g}$
  satisfies $$ \sup_{B_{g}(p, 2\pi)}K_{g}\leq 1,$$ and there is a
   submanifold $L$ calibrated by   $\Theta$ such that $ \dim_{\mathbb{R}}L=n$,  $p\in L$, and  $$
  \int_{L}\Theta <
 \frac{\pi}{2n}\varpi_{n-1} ,$$ where $ \varpi_{n-1}$ is the volume of $S^{n-1}$ with the standard metric of  constant curvature 1.
   Then the
  injectivity radius $i_{g}(p) $ of $ (M,g)$  at $p$ satisfies that $$i_{g}(p)^{n}\leq \frac{n\pi^{n-1}}{2^{n-1}\varpi_{n-1}} \int_{L}\Theta. $$
\end{corollary}

\begin{proof}[Proof of Corollary \ref{6.4} and  Theorem  \ref{0.4}] By Theorem \ref{6.04}, we have   $$  {\Vol}_{h_{1}}(B_{h_{1}}(r))\leq {\Vol}_{g}(B_{g}(p,r)\cap L)
\leq  {\Vol}_{g}( L) = \int_{L}\Theta,$$  for any  $r\leq \min
\{i_{g}(p), \frac{\pi}{2}\}$, where  $h_{1}$ denotes the standard
metric on
  $S^{n} $ with constant curvature 1, and $B_{h_{1}}(r)$ denotes a
  metric $r$-ball in  $S^{n} $.
Since $h_{1}=dr^{2}+ \sin^{2}rh_{S^{n-1}}$ where $h_{S^{n-1}}$ is
the standard metric on $S^{n-1}$ with constant curvature 1, we
obtain $\sin r \geq \frac{2}{\pi}r$, and
$$ \frac{2^{n-1}}{n\pi^{n-1}}r^{n}\varpi_{n-1}\leq \int_{0}^{r}\sin^{n-1}rdr\varpi_{n-1}=
 {\Vol}_{h_{1}}(B_{h_{1}}(r))\leq \int_{L}\Theta.$$

If $i_{g}(p)\geq \frac{\pi}{2}$, by letting $r=\frac{\pi}{2}$, we
obtain $$\frac{\pi}{2n}\varpi_{n-1} \leq \int_{L}\Theta <
\frac{\pi}{2n}\varpi_{n-1} ,$$ which is a contradiction. Thus
$i_{g}(p)< \frac{\pi}{2}$.  By letting $r=i_{g}(p)$, we obtain
$$i_{g}(p)^{n}\leq \frac{n\pi^{n-1}}{2^{n-1}\varpi_{n-1}} \int_{L}\Theta.$$

We obtain  Theorem  \ref{0.4} by applying  the above arguments  to
special lagrangian submanifolds in Ricci-flat  Calabi-Yau
  manifolds.
\end{proof}

  An  obvious application of Corollary \ref{6.4} is
  to estimate injectivity radiuses by volumes of holomorphic  submanifolds,  which has independent
  interests.

\begin{corollary}\label{6.5}
 Let $(M,  \omega, J, g)$ be   a closed  K\"{a}hler  m-manifold,  and $p\in M$. Assume that the sectional curvature $K_{g}$
  satisfies $$ \sup_{M}K_{g}\leq 1,$$ and there is a
  smooth  holomorphic n-submanifold   $N$ such that   $p\in N$, and  $$
\int_{N}\omega^{n}<
 \frac{(n-1)!\pi}{2}\varpi_{n-1}.$$
   Then the
  injectivity radius $i_{g}(p) $ of $ (M,g)$  at $p$ satisfies that
  $$i_{g}(p)^{n}\leq \frac{\pi^{n-1}}{(n-1)!2^{n-1}\varpi_{n-1}} \int_{N}\omega^{n}. $$
\end{corollary}

By combining this corollary and the  result in  \cite{CG2},  there
are $F$-structures of positive rank on the regions of  K\"{a}hler
 manifolds with bounded curvature and fibred by holomorphic
 submanifolds with   small volumes.

Finally, we estimate the  volume of metric balls when they are
fibred by calibrated submanifolds with small volume in the  absence
of  bounded curvature of the ambient manifolds.

\begin{theorem}\label{6.6} For any $n, m \in \mathbb{N}$ and any $\varepsilon >0 $,  there is a constant $\delta = \delta(n, m,  \varepsilon)>0$
satisfying that: Assume that  $(M,  g, \Theta)$ is  a closed
  Ricci-flat Einstein
 m-manifold with a calibration  n-form $\Theta$, $p\in M$,  and there is a  homology class  $A\in H_{n}(M,
 \mathbb{Z})$ such that  for any $x\in B_{g}(p,1)$, there is a n-submanifold $L_{x}$ calibrated by   $\Theta$   passing $x$ and presenting $A$, i.e.
 $x\in L_{x}$ and $[L_{x}]=A$. If  $$\int_{A}\Theta < \delta, $$ then
 $${\Vol}_{g}(B_{g}(p,1))\leq \varepsilon .$$
\end{theorem}

 \begin{proof}[Proof of Theorem   \ref{6.6} and  Theorem  \ref{0.5}]
Assume that  the conclusion  is not
     true.  Then  there is a  family of
      closed
  Ricci-flat Einstein
 $m$-manifolds $\{(M_{k},  g_{k}, \Theta_{k})\}$  with
  calibration  $n$-forms $\Theta_{k}$, $p_{k}\in M_{k}$,  and there are  homology class  $A_{k}\in H_{n}(M_{k},
 \mathbb{Z})$ such that  \begin{itemize}
   \item[i)] for any $x\in B_{g_{k}}(p_{k},1)$, there is a $n$-submanifold $L_{k,x}$ calibrated by   $\Theta_{k}$   passing $x$ and presenting $A_{k}$, i.e.
 $x\in L_{k,x}$ and $[L_{k,x}]=A_{k}$,
    \item[ii)] $$\int_{A_{k}}\Theta_{k} < \frac{1}{k}. $$ \end{itemize}
 However,  $${\Vol}_{g_{k}}(B_{g_{k}}(p_{k},1))\geq C,$$ for a
 constant $C>0$ independent of $k$.

  By Gromov pre-compactness  theorem  (Theorem \ref{2002}), by passing to a subsequence, $\{(M_{k},  g_{k},
  p_{k})\}$ converges to a  complete path metric space $(Y, d_{Y}, p)$  in the pointed
  Gromov-Hausdorff sense.  By  Theorem \ref{2003}, the Hausdorff dimension of $(Y, d_{Y},
  p)$ is $m$, and there is a closed  subset $S_{Y}\subset Y$ of Hausdorff
  dimension smaller than $m-1$, i.e. $\dim_{\mathcal{H}}S_{Y}<m-1 $,
  such that $Y\backslash S_{Y}$ is a smooth manifold, and $d_{Y}$ is
  induced by a Ricci-flat Einstein metric $g_{\infty}$ on $Y\backslash
  S_{Y}$.  Furthermore,  for any compact subset $D\subset Y\backslash
  S_{Y}$, there are embeddings $F_{k,D}: D \longrightarrow M_{k}$ such
  that $ F_{k,D}^{*}g_{k}$ converges to $g_{\infty}$ in the
  $C^{\infty}$-sense.

  We take $D$ big  enough such that $D\cap B_{d_{Y}}(p, \frac{1}{8})$ is not empty.
    For  a $y\in {\rm int}D\cap B_{d_{Y}}(p, \frac{1}{4})$, there is a $r>0$
    such that $B_{g_{\infty}}(y, 2r)\subset {\rm int}D\cap B_{d_{Y}}(p, \frac{1}{4})
    $, and $ B_{g_{k}}(F_{k,D}(y), r)\subset F_{k,D}(D)\cap B_{g_{k}}(p_{k},\frac{1}{2})$ when  $k\gg
    1$.  By the smooth convergence of $  F_{k,D}^{*}g_{k}$,  the
    sectional curvatures $K_{g_{k}}$ satisfy $$ \sup_{F_{k,D}(D)}|K_{g_{k}}|\leq C_{D},
    $$ for a constant $C_{D}$ depending on $D$, but independent of
    $k$.  Let $L_{k}$ be a  $n$-submanifold   calibrated by   $\Theta_{k}$   passing $F_{k,D}(y)$ and presenting $A_{k}$, i.e.
 $F_{k,D}(y)\in L_{k}$ and $[L_{k}]=A_{k}$. By Bishop-Gromov
 comparison theorem, for any $\rho  \leq r $, $$ \frac{{\Vol}_{g_{k}}(B_{g_{k}}(F_{k,D}(y),\rho))}{\rho^{m}}\geq
 {\Vol}_{g_{k}}(B_{g_{k}}(F_{k,D}(y), 1))\geq
 {\Vol}_{g_{k}}(B_{g_{k}}(p_{k}, \frac{1}{2})) $$ $$\geq \frac{1}{2^{m}}{\Vol}_{g_{k}}(B_{g_{k}}(p_{k}, 1))\geq \frac{C}{2^{m}}. $$
   Then there is a uniform lower bound $ \iota >0$ for injectivity radiuses
   $i_{g_{k}}(F_{k,D}(y))$ at $F_{k,D}(y)$ (c.f. \cite{Pe}), i.e.
   $i_{g_{k}}(F_{k,D}(y))\geq \iota >0
   $ for $k\gg 1$. By Theorem  \ref{6.04},  we have $${\Vol}_{h_{1}}(B_{h_{1}}(\rho))\leq {\Vol}_{g_{k}}(B_{g_{k}}(F_{k,D}(y),\rho)
   \cap L_{k}) \leq \int_{L_{k}}\Theta_{k} = \int_{A_{k}}\Theta_{k}
   \leq\frac{1}{k},$$ where $\rho =\min \{\iota, r,
   \frac{\pi}{\sqrt{C_{D}}}\}$,  $h_{1}$ denotes the standard  metric on
  $S^{n} $ with constant curvature $C_{D}$ , and $B_{h_{1}}(\rho)$ denotes a
  metric $\rho$-ball in  $S^{n} $.  We obtain a contradiction when    $k\gg
  1$, and, thus, we obtain the conclusion.

  We obtain  Theorem  \ref{0.5} by applying  the above arguments  to
special lagrangian submanifolds in Ricci-flat  Calabi-Yau
  manifolds.
  \end{proof}

\begin{remark}  Let  $\{(M_{k},  g_{k}, \Theta_{k})\}$ be  a
family of closed  Ricci-flat Einstein
 $m$-manifolds with calibration $n$-forms $\Theta_{k} $,  and $\{p_{k} \}$ be a sequence of points such that
 there are open subsets  $W_{k}\supset B_{g_{k}}(p_{k},1)$ in $M_{k}$  admitting
  calibrated
   fibrations  $f_{k}: W_{k} \longrightarrow B_{k}$
    corresponding to   $\Theta_{k}$, i.e.  for any $b_{k}\in B_{k}$,
    $f_{k}^{-1}(b_{k})$ is a $n$-submanifold calibrated by
    $\Theta_{k}$.
   If  $$ \lim_{k\longrightarrow\infty} \int_{f_{k}^{-1}(b_{k})}\Theta_{k}=0, $$
   where $b_{k} \in B_{k}$,  Theorem  \ref{6.6}  implies that $$\lim_{k\longrightarrow\infty}
 {\Vol}_{g_{k}}(B_{g_{k}}(p_{k},1)) =0,$$ and, by passing to  a subsequence, $\{(M_{k},  g_{k})\}$
 converges  to a  path metric space of lower Hausdorff  dimension in the pointed Gromov-Hausdorff sense.
\end{remark}
\begin{remark}  Theorem  \ref{6.6} can be applied to
$G_{2}$-manifolds and $Spin(7)$-manifolds since they are Ricci-flat
Einstein manifolds (c.f. \cite{J1}).
\end{remark}

\end{document}